\documentclass[10pt]{amsart}
\usepackage{graphicx, xcolor}
\usepackage{amssymb}
\numberwithin{equation}{section}

\def\R{\mathbb{R}}

\newcommand{\rdbrack}{]\!]}
\newcommand{\ldbrack}{[\![}

\def\g{\gamma}

\def\k{\kappa}

\def\epsilon{\varepsilon}
\def\e{\varepsilon}

\newcommand\br{\begin{rem}}
\newcommand\er{\end{rem}}
\newcommand\bp{\begin{pmatrix}}
\newcommand\ep{\end{pmatrix}}
\newcommand\be{\begin{equation}}
\newcommand\ee{\end{equation}}
\newcommand\ba{\begin{equation}\begin{aligned}}
\newcommand\ea{\end{aligned}\end{equation}}

\newtheorem{theorem}{Theorem}[section]
\newtheorem{proposition}[theorem]{Proposition}
\newtheorem{lemma}[theorem]{Lemma}

\newtheorem{example}[theorem]{Example}
\newtheorem{remark}[theorem]{Remark}
\newtheorem{ans}[theorem]{Definition}

\title{Stability properties of the steady state for the isentropic compressible  Navier-Stokes equations with density dependent viscosity in bounded intervals} 

\begin{document}

\maketitle

\begin{center} 
 MARTA STRANI\footnote{Universit\`a Ca' Foscari, Dipartimento di Scienze Molecolari e Nanosistemi, Venezia Mestre (Italy),
 E-mail address: \texttt{marta.strani@unive.it, martastrani@gmail.com}. }
\end{center}
\vskip1cm

\begin{abstract} We prove existence and asymptotic stability of the stationary solution for the compressible Navier-Stokes equations for isentropic gas dynamics with a density dependent  diffusion  in a bounded interval.  We present the necessary conditions to be imposed on the boundary data which ensure existence and uniqueness of the steady state, and we subsequent investigate its stability properties by means of the construction of a suitable Lyapunov functional for the system. The Saint-Venant system, modeling the dynamics of a shallow compressible fluid, fits into this general framework.
\end{abstract}

\begin{quote}\footnotesize\baselineskip 14pt 
{\bf Key words.} 
Navier-Stokes equations, parabolic-hyperbolic systems, stationary solutions, stability.
 \vskip.15cm
\end{quote}

\begin{quote}\footnotesize\baselineskip 14pt 
{\bf AMS subject classification.} 
35Q35, 35B35, 35B40, 76N10.
 \vskip.15cm
\end{quote}

\pagestyle{myheadings}
\thispagestyle{plain}
\markboth{M.STRANI}{STABILITY OF THE STEADY STATE FOR THE COMPRESSIBLE NAVIER-STOKES EQUATIONS }

\section{Introduction}

In this paper we study existence and stability properties of the steady state for the one dimensional compressible Navier-Stokes equations with density-dependent viscosity, which describes the isentropic motion of compressible viscous fluids {in a bounded interval}. In terms of
the variables mass density and velocity of the fluid $(\rho, w)$,  the problem reads as
\begin{equation}\label{PLgenerico2}
 \left\{\begin{aligned}
&  \rho_t + (\rho\, w)_x=0 \\
& (\rho \, w)_t+  \left( \rho \, w^2+ P(\rho)\right)_x= \varepsilon\left( \nu(\rho)  w_x \right)_x,  \ \ \ \ \ \ \ x\in(-\ell, \ell) 
   \end{aligned}\right. 
\end{equation}
to be complemented with  boundary conditions
\begin{equation*}
\rho(-\ell)=\rho_->0, \quad w(\pm \ell)= w_\pm >0,
\end{equation*}
and initial data $(\rho,w)(x,0)=(\rho_0(x),w_0(x))$, with $\rho_0>0$.

We restrict our analysis to the barotropic regime, where the pressure $P \in C^2(\R^+)$ is  a given function of
 the density $\rho$ satisfying the following assumptions
\begin{equation}\label{ipotesisuPgen}
P(0)=0, \quad P(+\infty)=+\infty, \quad P'(s),P''(s)>0 \ \ \  \forall s >0.
\end{equation}
Since it is well know that, in the isentropic case, the viscosity of a gas depends on the density,  we consider the case of a density dependent viscoity   $\nu \in C^1(\R^+)$ such that  $\nu(\rho)>0$ for all $\rho >0$. 

A prototype for the term of pressure  is given by the power law $P(\rho)=\rho^\gamma$, with $\gamma > 1$ the adiabatic constant, while the case $\nu(\rho)=\rho^{\alpha}$ with $\alpha>0$ is known as the Lam\'e viscosity coefficient.
In particular, the well known viscous Saint-Venant system, describing the motion of a shallow compressible fluid, corresponds to the choice $P(\rho)=\frac{1}{2}\kappa \rho^2$, $\kappa >0,$ and {$\nu(\rho)=\rho$}.

\vskip0.2cm
By considering the variables density and momentum $(u,v)=(\rho,\rho w)$, system \eqref{PLgenerico2} {becomes}
\begin{equation}\label{PLgenerico}
 \left\{\begin{aligned}
&  u_t + v_x=0 \\
&  v_t+  \left( \frac{v^2}{u}+ P(u)\right)_x= \varepsilon\left( \nu(u) \left(\frac{v}{u}\right)_x \right)_x, \\
   \end{aligned}\right. 
\end{equation}
together with boundary conditions
\begin{equation}\label{BC}
v_\pm u(-\ell)-v(\pm\ell)=0, \quad v(- \ell)= \rho_-w_- .
\end{equation}
In the following, we shall use both the formulations \eqref{PLgenerico2} and \eqref{PLgenerico}, depending on what is necessary; as an example, when studying the stationary problem, the variables $(u,v)$ appear to be more appropriate since in this case the second component of the steady state turns to be a constant.

\begin{remark}{\rm {Given an initial datum $(\rho_0,w_0) \in H^1(I) \times H^1(I)$, throughout the paper we will consider solutions to \eqref{PLgenerico2} $(\rho,w) \in L^\infty([0,T], H^1(I))\times L^\infty([0,T], H^1(I)) \cap L^2([0,T], H^2(I))$. We refer the readers to Theorem \ref{existence} in Appendix A  for the proof of the existence of such a solution. We stress that, of course,  also the solution $(u,v)$ to \eqref{PLgenerico} belongs to the same functional space,  since $(u,v)=(\rho, \rho w)$.}}
\end{remark}

Depending on assumptions and approximations,  the Navier-Stokes system may also contain
other terms and gives raise to different types of partial differential equations. Indeed, natural
modifications of the model emerge when additional physical effects are taken into account,
like viscosity, friction or Coriolis forces; far from being exhaustive but only intended to give a small flavor of the huge number of references, see, for instance, \cite{CK, Des, FeiPet, FeiPet2, LZZ} and the references therein for existence results of global weak and strong solutions,  \cite{JWX, LXY, MelVas} for the problem with a density-dependent  viscosity  vanishing on vacuum, \cite{BreDes07}  for the full Navier-Stokes system for viscous compressible and heat conducting fluids.   More recently, the interesting phenomenon of metastability
has been investigated both for the incompressible model \cite{BW, LinXu} and for the $1$D compressible problem \cite{S16}. 
\vskip0.2cm
As concerning system \eqref{PLgenerico2}, there { is a vast literature}  in both one and higher dimensions. Global existence results and asymptotic stability of the equilibrium states are obtained from Kawashima's theory of parabolic-hyperbolic systems in \cite{Kawa87}, D. Bresch, B. Desjardins and G. M\'etivier  in \cite{BrDeMe06}, P.L. Lions in \cite{Lio9698} and W. Wang in \cite{WangXu05} for the  viscous model, and C.M. Dafermos (see \cite{Daf97}) for the inviscid model.  
  As the compressible Navier-Stokes equations with density-dependent viscosity are suitable to model the dynamics of a  compressible viscous flow in the appearance of vacuum \cite{HS}, there are many literatures on the well-posedness theory of the solutions and their asymptotic behavior  for the $1$D model (see, for instance, \cite{CZ, GZ, JXZ, KV, MelVas2, QHY, YYZ, YZ} and the references therein).
However, most of these results concern with free boundary conditions. Recently, the analysis of the dynamics in bounded domains has also been  investigated (see, for instance, \cite{BDG}); the initial-boundary value problem with $\nu(\rho) = \rho^\alpha$,  $\alpha > 1/2$,  has been studied by H.L. Li, J. Li and Z. Xin in \cite{LiLiXin08}: here the authors are concerned with  the phenomena
of finite time vanishing of vacuum. We also quote the analysis  performed in \cite{LianGuoLi10}, where a particular attention is devoted to the dynamical  behavior close to equilibrium configurations.

\vskip0.2cm
The existence of stationary solutions for system \eqref{PLgenerico2} and, in particular, for  shallow water's type systems, and the subsequent investigation of their stability properties has also been considered in the literature.  To name some of these results, we recall here \cite{BasCorAN} and \cite{DiBaCo}, where the authors are concerned with the inviscid case; in particular, in \cite{DiBaCo} the authors  address the issue of stating sufficient boundary
conditions for the exponential stability of linear hyperbolic
systems of balance laws (for the investigation of the nonlinear problem, we refer to \cite{CBA, CBA2}).
\vskip0.2cm
The case with real viscosity has been addressed, for example, in  \cite{MascRou06};  we mention also the recent contributions \cite{KNZ, MN}, where the authors investigate asymptotic stability of the steady state in the half line. We point out that when dealing with the open channel case (i.e. $x\in \mathbb{R}$), the study of  the stationary problem is different than the one in the case of  bounded domains, where one has to handle compatibility conditions on the boundary values coming from the study of the formal hyperbolic limit $\e=0$. In this direction, we quote the papers \cite{PS,STRAS}, where the authors address the problem of the long time behavior of solutions  for the Navier-Stokes system in one dimension and with Dirichlet boundary conditions (see also \cite{STRAS2} for the extension of the results to the case of a density dependent diffusion). Finally, a recent contribution in the study of the stationary problem associated with a  simplified version of \eqref{PLgenerico2}
 is the paper \cite{Str14}, where the author considers  the special case of $P(\rho)=\kappa \,\rho^2/2$ and {$\nu(\rho)=\rho$}, corresponding to the viscous shallow water system.  

Being the literature on the subject so vast, we are aware that this list of references is far from being  exhaustive. 

\vskip0.2cm
Our aim in the present paper is  at first to prove existence and uniqueness of a stationary solution to \eqref{PLgenerico}-\eqref{BC}. Because of the discussion above, this results is likely to be achieved only if some appropriate assumptions on the boundary values are imposed; precisely, following the line of \cite{Str14} where this problem has been addressed for the case of a linear diffusion, our first main contribution (for more details, see Section 3), is the following theorem.
\begin{theorem}\label{teointro}
Given $\ell>0$ and $v_-, u_\pm>0$, let us consider the  problem
\begin{equation}
 \left\{\begin{aligned}\label{EqTeo}
 & u_t +  v_x=0 &\qquad &x \in (-\ell,\ell), \ t \geq 0 \\
& v_t+\left\{ \frac{v^2}{u}+ P(u)-\varepsilon \, \nu(u) \left( \frac{v}{u}\right)_x\right\}_x=0   \\
& u(\pm \ell,t)= u_\pm, \quad v(- \ell, t)=v_-, &\qquad &t \geq 0 \\
& u(x,0)=u_0(x), \quad v(x,0)=v_0(x) &\qquad &x \in (-\ell,\ell),
\end{aligned}\right.
\end{equation}
and let us suppose that the following assumptions are satisfied:
\vskip0.2cm
{\bf H1.} The term of pressure $P(u) \in C^2(\R^+)$ and the viscosity term $\nu(u)$ verify, for all $u>0$
\begin{equation*}
P(0)=0, \quad P(+\infty)=+\infty, \quad P'(u) \quad {\rm and } \quad P''(u)>0, \qquad \nu(u)>0;
\end{equation*}
\vskip0.1cm
{\bf H2.} Setting $f(u):= u\, \int_0^u P(z)/z^2 \, dz$, there hold 
 $$v_*^2 \, (u_+-u_-)= u_- \,u_+ \, \left(P(u_+)-P(u_-)\right)\quad {\rm and} \quad \left[\!\!\left[ \frac{v^3}{2u^2}+v f'(u)\right] \!\!\right] \leq 0,
$$
where $\ldbrack \,\cdot \, \rdbrack$  denotes the jump  and where $v_* := v_- \equiv v_+$. Then there exists a unique stationary solution $(\bar u(x),\bar v(x))$ to \eqref{EqTeo}, i.e. a unique solution $(\bar u,\bar v)$ independent on the time variable $t$ to the following boundary value problem
\begin{equation}
 \left\{\begin{aligned}\label{EqTeo2}
 & v_x=0 \\ 
& \left\{ \frac{v^2}{u}+ P(u)-\varepsilon \, \nu(u) \left( \frac{v}{u}\right)_x\right\}_x=0,   \\
& u(\pm \ell)= u_\pm, \quad v(- \ell)=v_-.  \\
\end{aligned}\right.
\end{equation}

\end{theorem}

\begin{remark}{\rm
We stress that the choice of the variables $(u,v)$ (instead of the most common choice density/velocity), is dictated by the fact that the second component  of the steady state turns to be a constant, and this constant value is univocally determined once the boundary data are imposed.

}
\end{remark}

Once the existence of a unique steady state for system \eqref{EqTeo} is proved, we devote the second part of this paper to investigate its stability properties. Precisely,  we prove stability of the steady state in the sense of the following definition ($L^2$-stability).
\begin{ans}\label{defintro}
A  stationary solution $(\bar u,\bar v)$ to \eqref{EqTeo} is {\it stable} if for any $\e_0>0$ there exists $\delta_0=\delta_0(\e_0)$ such that, if $|(u_0,v_0)(x)-(\bar u,\bar v)(x)|_{{}_{L^2}} < \delta_0$, then, for all $T>0$, it holds $$\sup_{0\leq t\leq T}|(u,v)(t)-(\bar u,\bar v)|_{{}_{L^2}} \leq \e_0, $$
where $(u,v)(t)$ is the solution to  \eqref{EqTeo}.
\end{ans}
Our second main result is stated in the next
theorem; it shows the  stability of the
stationary solution constructed in Theorem \ref{teointro} under an additional assumption on the values of the density and its derivative at the boundary.
\begin{theorem}\label{teointro2}
Let the assumptions of Theorem \ref{teointro} be satisfied, 
and let us also assume that there exist constants $\delta_1, \delta_2>0$ small enough such that   the boundary values $u(\pm\ell, t)=u_\pm$ and $u_x(\pm \ell,t)$ satisfy
\vskip0.2cm
{\bf H3.}
$|u_+-u_-| < \delta_1 \qquad $ and $\qquad |u_x(\ell,t)-u_x(-\ell,t)| < \delta_2 \ \ \forall \ t \geq 0.$

\vskip0.2cm
\noindent Then the steady state $(\bar u,\bar v)$ is stable in the sense of Definition \eqref{defintro}.
\end{theorem}

\begin{remark}
{\rm It is worth notice that Theorem \ref{teointro2} prove stability of the steady state for all time (since the constant $\e_0$  in Definition \ref{defintro} in independent on $T>0$). We also point out that the strategy used here do not provide stability of $(\bar u, \bar v)$ in the case the  boundary values $u_\pm$ and $u_x(\pm \ell)$ do not satisfy any smallness condition, while its existence is assured also in this setting (cfr Theorem \ref{teointro}); however, our guess is that \textquotedblleft large\textquotedblright \, solutions are not stable (see also the analysis of \cite{MN} and \cite{ZZZ}, where a similar smallness condition has been imposed  in order to have stability of the steady state to a Navier-Stokes system in the half line), and this will be the object of further investigations. We finally notice that, because of the results of Appendix A (in particular, see Remark \ref{normapiccola}), if the initial datum for the density satisfies $|\rho_0|_{{}_{L^\infty}} \leq \delta_1$, then this property is invariant under the dynamics.


}
\end{remark}

We close this Introduction with a short plan of the paper. In Section 2 we study the {inviscid} problem, obtained formally by setting $\e=0$ in \eqref{PLgenerico}; we show that, at the hyperbolic level, some compatibility conditions on the boundary data are needed in order to ensure the existence of a weak solution. In particular, such conditions follow from the definition of a couple entropy/entropy flux which, in the present setting, are given by
\begin{equation*}
\mathcal E(u,v):=\frac{v^2}{2u}+f(u) \quad \mbox{and} \quad  \mathcal Q(u,v):=\frac{v}{u}\left[ \frac{v^2}{2u}+u f'(u)\right],
\end{equation*}
being $f(u):= u \int _{0}^u P(z)/ z^2 \, dz$.

Section 3 is devoted to the study of the stationary problem  for \eqref{EqTeo}, and in particular to the proof  of Theorem \eqref{teointro}; to this aim we will state and prove several Lemmas  showing that, once the boundary conditions are imposed and assumption {\bf H2} is satisfied, there exists a {\it unique positive connection} for \eqref{EqTeo}, i.e. a unique stationary solution connecting the boundary data. 
Such analysis deeply relies on the strategy firstly performed in \cite{Str14}, where the author addresses the same problem in the easiest case of a linear diffusion, namely $\nu(u)=u$.

Section 4 is the core of the paper, and we investigate to the stability properties of the steady state, proving that it is stable in the sense of Definition \ref{defintro}; the key point to achieve such result is the construction of a Lypaunov functional, which, in the present setting, is defined as
\begin{equation*}
L(u(t),v(t),\bar u,\bar v) :=  \int_{-\ell}^\ell \frac{(v-\bar v)^2}{2u} + u \, \psi (u, \bar u) \, dx, \quad \psi (u, \bar u)= \int_{\bar u}^u \frac{P(z)-P(\bar u)}{z^2} \, dz.
\end{equation*}
It is easy to check that $L(u,v,\bar u,\bar v)$ is positive defined and null only when computed on the steady state; the tricky part will be the computation of the sign of its time derivative along the solutions, needed in order to apply a Lyapunov type stability theorem. 

Finally, in Appendix A we  prove the existence of a solution  to \eqref{PLgenerico2} belonging to $L^{\infty}([0,T],H^1(I))$ $\times L^{\infty}([0,T],H^1(I))\cap L^{2}([0,T],H^2(I))$.

\vskip0.2cm
As stressed in the introduction, results relative to the existence and stability properties of the steady state for the Navier-Stokes equations in bounded intervals appear to be rare; the study of the stationary problem for \eqref{EqTeo} (with  generic pressure $P(u)$ and  viscosity $\nu(u)$) and, mostly,  the subsequent investigation of the stability properties of the steady state are, to the best of our knowledge, new.


\section{The inviscid  problem }

We start our analysis by studying the limiting  regime $\e \to 0$, obtained formally by putting $\e=0$ in \eqref{PLgenerico2}; we obtain the following hyperbolic system for unviscous isentropic fluids
\begin{equation}\label{PLepszero1}
 \left\{\begin{aligned}
 &  \rho_t + (\rho w)_x=0, \\
&  (\rho w)_t+\left( \rho w^2+ P(\rho)\right)_x=0.
\end{aligned}\right.
\end{equation}
System \eqref{PLepszero1} is complemented with the same boundary and initial conditions of \eqref{PLgenerico2}.
We recall that the usual setting where such a system is studied  is given by the \textit{entropy formulation}, hence non classical discontinuous solutions can appear; thus, we primarily concentrate on the problem of determining the entropy jump conditions for the hyperbolic system \eqref{PLepszero1}.
As previously done in \cite{Str14}  such conditions are dictated by the choice of a couple entropy/entropy flux $\mathcal  E=\mathcal E(\rho,w)$ and $\mathcal Q= \mathcal Q(\rho,w)$ such that
\begin{itemize}
\item the mapping $(\rho,w) \to \mathcal E$ is convex;
\vskip0.2cm
\item $ \mathcal E_t + \mathcal Q_x =0$ in any region where $(\rho,w)$ is a solution to \eqref{PLepszero1}.
\end{itemize}
In particular,  $ \mathcal E_t +\mathcal Q_x \equiv0$ if and only if
\begin{equation}\label{eqflux}
\left\{ \begin{aligned}
 \mathcal Q_\rho &= w \,  \mathcal  E_\rho + \frac{P'}{\rho} \, \mathcal E_w \\
 \mathcal Q_w &= \rho \,  \mathcal  E_\rho + w \,   \mathcal E_w.
\end{aligned}\right.
\end{equation}
In the case of a general term of pressure $P(\rho)$ satisfying assumptions \eqref{ipotesisuPgen}, the couple entropy/entropy flux is given by
\begin{equation}\label{EQgeneric}
\mathcal E(\rho,w)=\frac{1}{2}\rho w^2 + f(\rho) \qquad \mbox{and} \qquad  \mathcal Q(\rho,w) =w\left[\frac{1}{2}\rho w^2+\rho \,f'(\rho)\right], 
\end{equation}
being 
\begin{equation*}
f ( \rho ) = \rho \int^{\rho}_0 \frac{p ( z )}{ z^2} \, dz.
\end{equation*}
When considering the case of a power law type of pressure, i.e. $P(\rho)=\kappa \rho^\gamma$ with $\kappa >0$ and $\gamma >1$, the entropy corresponds to the physical energy of the system  (see, for instance, \cite{Evans}) and it is defined as
\begin{equation}\label{E}
\mathcal E(\rho,w):=\frac{1}{2}\rho w^2+ \frac{\kappa}{\gamma-1} \rho^{\gamma}.
\end{equation}
By solving \eqref{eqflux}, it turns out that  $\mathcal Q$ is defined as
\begin{equation}\label{Q}
\mathcal Q(\rho, w)= w \left[ \frac{1}{2}\rho w^2 +\frac{\kappa \gamma}{\gamma-1} \rho^\gamma\right],
\end{equation}
and we recover \eqref{EQgeneric} with $P(\rho)=\kappa \rho^\gamma$. 

Following the line of \cite{Str14},  and given $\rho_\pm >0$, $w_\pm>0$ and $c \in \R$, we let $(\rho_-,w_-)$ and $(\rho_+,w_+)$ be an entropic discontinuity of \eqref{PLepszero1} with speed $c$, that is we assume  the function
\begin{equation}\label{weaksol}
 (\Upsilon,W)(x,t) :=  \left\{\begin{aligned}
&(\rho_-,w_-) \quad \textrm{for} \ x<ct \\
&(\rho_+,w_+) \quad \textrm{for} \ x>ct
   \end{aligned}\right.
\end{equation}
to be a weak solution to \eqref{PLepszero1} satisfying, in the sense of distributions, the entropy inequality
\begin{equation}\label{PNEineq}
\frac{\partial \mathcal E}{\partial t}+ \frac{\partial \mathcal Q}{\partial x} \leq 0.
\end{equation}
On  one side, with the change of variable $\xi=x-ct$, system \eqref{PLepszero1} reads
\begin{equation*}
 \left\{\begin{aligned}
 &  -c \, \rho_\xi +(\rho w)_\xi=0, \\
& -c  (\rho w)_\xi +  \left(\rho w^2 +P(\rho)\right)_\xi=0,  
\end{aligned}\right.
\end{equation*}
and the request of weak solution translates into the \textit{Rankine-Hugoniot} conditions, that read
\begin{equation}\label{ranhug}
\ldbrack\rho(w-c)\rdbrack=0 \qquad \mbox{and} \qquad \ldbrack\rho w (w-c)+P(\rho)\rdbrack=0.
\end{equation}
On the other side, the {\it entropy condition} \eqref{PNEineq} reads  $\ldbrack \mathcal Q -c \mathcal E\rdbrack \leq 0$, where
\begin{equation*}
\mathcal Q - c\mathcal E =\frac{1}{2}\rho w^3+f'(\rho) w \rho-\frac{c}{2}\rho w^2-cf(\rho).
\end{equation*}
Setting $w-c=z$, we have
\begin{equation*}
\mathcal Q - c\mathcal E =\frac{1}{2}\rho z^3+f'(\rho)\rho z+c\left[\rho v^2-f(\rho)+f'(\rho)\rho\right]+\frac{1}{2}c^2 \rho v,
\end{equation*}
so that, recalling 
\begin{equation*}
f'(\rho)= \int_0^{\rho} \frac{P(z)}{z^2} \, dz+ \frac{P(\rho)}{\rho} \quad \Longrightarrow \quad f'(\rho)\rho=P(\rho)+f(\rho),
\end{equation*}
and by using \eqref{ranhug} , the entropy condition translates into
\begin{equation}\label{PNEineq2}
 \left[\!\!\left[  \frac{1}{2}\rho z^3 +\rho z\, f'(\rho)   \right]\!\!\right]  \leq 0.
\end{equation}
By squaring the first equation in \eqref{ranhug}, we obtain a system for the quantities $z_{\pm}^2$, whose solutions are given by
\begin{equation}\label{h}
z_+^2= \frac{\rho_-}{\rho_+}\frac{[P(\rho_+)-P(\rho_-)]}{(\rho_+-\rho_-)},   \qquad z_-^2=\frac{\rho_+}{\rho_-}\frac{[P(\rho_+)-P(\rho_-)]}{(\rho_+-\rho_-)}.
\end{equation}
When looking for the stationary solutions to \eqref{PLepszero1}, i.e. $c=0$, \eqref{h} translates into the following conditions for the boundary values
\begin{equation}\label{condizionialbordo}
w_+^2= \frac{\rho_-}{\rho_+}\frac{[P(\rho_+)-P(\rho_-)]}{(\rho_+-\rho_-)}   \qquad \mbox{and} \qquad w_-^2=\frac{\rho_+}{\rho_-}\frac{[P(\rho_+)-P(\rho_-)]}{(\rho_+-\rho_-)},
\end{equation}
that, together with \eqref{PNEineq2}, univocally determine the possible choices of  the boundary  data for the jump solution \eqref{weaksol} with $c=0$ to be an admissible steady state for the system. 

In particular, for all $x_0 \in (-\ell,\ell)$, we can state that the one-parameter family 
\begin{equation*}
(\Upsilon,W)(x) = (\rho_-, w_-) \chi_{_{(-\infty,x_0)}} + (\rho_+, w_+) \chi_{_{(x_0, \infty)} }
\end{equation*}
is a family of stationary solutions to \eqref{PLepszero1}
if and only if  both \eqref{PNEineq2} and \eqref{condizionialbordo} are satisfied.

Finally, we point out   that, in terms of the variables density/momentum,  conditions \eqref{PNEineq2}-\eqref{condizionialbordo} read as
\begin{equation*}
v_+^2=v_-^2= u_+u_-\frac{[P(u_+)-P(u_-)]}{(u_+-u_-)} \qquad \mbox{and} \qquad \left[\!\!\left[  \frac{v}{u}\left(\frac{v^2}{2u} +u f'(u)\right)   \right]\!\!\right] \leq 0.
\end{equation*}

\begin{example}
{\rm
In the case of the scalar Saint-Venant system, i.e. $P(u)=\frac{1}{2}\kappa u^2$, stationary solutions to 
\begin{equation*}
  u_t + v_x=0, \quad
  v_t+\left( \frac{v^2}{u}+ P(u)\right)_x=0,  
\end{equation*}
to be considered with boundary data $u(\pm\ell,t)=u_\pm$ and $v(-\ell)=v_*$,
solve
\begin{equation*}
v=  v_*, \quad \frac{1}{2}\kappa u^3-\alpha u + v_*^2=0,
\end{equation*}
where
\begin{equation*}
v_*^2=\frac{1}{2}\kappa \, u_-u_+(u_++u_-) \qquad {\rm and} \qquad \alpha=\frac{1}{2}\kappa \,(u_+^2+u_+u_-+u_-^2).
\end{equation*}
Moreover, only entropy solutions are admitted, so that, from \eqref{PNEineq2} 
\begin{equation*}
\frac{v_+}{u_+}(u_+-u_-) \geq 0.
\end{equation*}
Since $v_+,u_+ >0$, then $u_-<u_+$, and this condition describes the realistic phenomenon of the \textit{hydraulic jump} consisting in an abrupt rise of the fluid surface and a corresponding decrease of the velocity. 
}
\end{example}

\section{Stationary solutions for the viscous problem}

This section is devoted to the study of  the existence and uniqueness of a stationary solution for the Navier-Stokes system \eqref{PLgenerico}. As stressed in the introduction,  we here prefer to use the variables density/momentum $(u,v)$ rather than the most common choice density/velocity since in this case the second component of steady state turns to be  a constant, which is univocally determined by the boundary values. We are thus left with a single equation for the variable $u$ which can be integrated with respect to $x$, by paying the price of the appearance of an integration constant. 
\vskip0.2cm
For $\varepsilon >0$, the stationary equations read
\begin{equation}\label{PLstaz}
\begin{aligned}
v_x=0, \qquad \left\{ \frac{v^2}{u}+ P(u)- \varepsilon \nu(u) \left(\frac{v}{u}\right)_x \right\}_x=0, 
  \end{aligned}
\end{equation}
which is a couple of ordinary differential equations; by integrating in $x$ we can lower the order of the system obtaining the following stationary problem for the couple $(u,v)$:
\begin{equation}\label{PLgenericostaz}
 \left\{\begin{aligned}
& v=v_* \\
&v_* \varepsilon \frac{\nu(u)}{u} u_x=-P(u)u+\alpha\, u-v_*^2 \\
& u(\pm \ell)=u_\pm, \quad v(- \ell)=v_-,
  \end{aligned}\right.
\end{equation}
being  $\alpha$  an integration constant that depends on the values of the solution and its derivative on the boundary, while $v_*$ is  univocally determined by the boundary datum $v(- \ell)$; indeed, since the component $v$ of the steady state turns to be  constant, the values $v(\ell)$ and $v(-\ell)$ are forced to be equal to a common value, named here $v_*$.

Let us define $$\Phi(u)=\int_0^u \frac{\nu(s)}{s}ds;$$ 
since $\nu>0$ there hold
\begin{equation*}
\Phi(u)>0 \quad \mbox{and} \quad \Phi'(u)= \frac{\nu(u)}{u} >0, \quad \forall u >0,
\end{equation*}
 where $\Phi'$ indicates the derivative of $\Phi$ with respect to $u$. Thus, the second equation in \eqref{PLgenericostaz} can be rewritten as
\begin{equation*}
v_* \varepsilon [\Phi( u)]_x=-P(u)u+\alpha\, u-v_*^2.
\end{equation*}
Setting $f(u):=-P(u)u+\alpha\, u-v_*^2$, with the change of variable $w=\Phi(u)$, and since $\Phi(u)$ is invertible, we have
\begin{equation}\label{autonoma}
v_* \varepsilon   w_x=(f \circ \Phi^{-1})(w) \equiv g(w).
\end{equation}
We thus end up with an autonomous  first order differential equation of the form $w'=g(w)$; in this case it is not possible to obtain an explicit expression for the solution, and in order to provide qualitative properties of the solution we have to study the function $g(w)$.

The problem of studying properties of the right hand side of  \eqref{autonoma} has been previously addressed in \cite{Str14} in the case of a linear diffusion $\nu(u)=u$ (that is, $\Phi(u) \equiv \mbox{Id}$). Precisely, the author states and proves a set of results describing the behavior of the function $f(u;\alpha)$ both with respect to $u$ and with respect to $\alpha$.

We recall here some of these results for completeness, since they will be useful to describe the qualitative properties of the function $g(w):=(f \circ \Phi^{-1})(w)$; for more details we refer to \cite[Lemma 3.1, Lemma 3.2]{Str14}. From now on, we will always suppose the pressure term $P(u)$ to satisfy assumptions \eqref{ipotesisuPgen}. We also recall that, by definition
\begin{equation}\label{definizionef}
f(u):=-P(u)u+\alpha\, u-v_*^2.
\end{equation}

\begin{lemma}\label{lemma-1}
For every $v_*>0$, there exists at least a value $\alpha$ such that there exist two positive solutions to the equation $f(u)=0$. 
\end{lemma}
\begin{remark}\label{R1}{\rm
As enlightened in the proof of \cite[Lemma 3.1]{Str14}, a sufficient condition on the constant $\alpha$ for the existence of two positive solutions to $f(u)=0$ is given by
\begin{equation}\label{sigmaeq1}
v_*< \sqrt{P'(u^*){u^*}^2},
\end{equation}
where $u^*=u^*( \alpha)$ solves $f'(u)=0$, while $v_*=v(\pm\ell)$. Indeed, since $f(0)=-v_*^2$ and $f(+\infty)=-\infty$, if $u^*$ is such that
\begin{equation*}
f(u^*)=\max_{\mathbb{R}^+}f, \quad f(u^*)>0,
\end{equation*} 
then the claim follows. By exploiting the conditions $f(u^*)>0$ and $f'(u^*)=0$, we end up  with \eqref{sigmaeq1}.

}
\end{remark}

\begin{lemma}\label{lemma0}
Let $\alpha$ be such that there exist two positive solutions $u_1<u_2$ to the equation $f(u)=0$. Hence, given $u_\pm >0$, the set $\mathcal A$ defined as
\begin{equation*}
\mathcal A := \{ \alpha >0 : u_1<u_-<u_+<u_2 \}
\end{equation*}
is such that $\mathcal A=[\bar \alpha,+\infty)$, for some $\bar \alpha>0$.
\end{lemma}

Lemmas \ref{lemma-1}-\ref{lemma0} assure that, once the boundary conditions $u_\pm$  are imposed, there always  exists a value for the integration constant $\alpha$ such that there exist two positive solutions $u_{1,2}$ to the equation $f(u)=0$  satisfying 
\begin{equation*}
(u_-,u_+) \subset (u_1,u_2).
\end{equation*}
This is of course a necessary condition for the existence of an increasing positive connection between  $u_-$ and $u_+$, as enlightened in Figure \ref{primafigura} in the specific example of $P(u)= \kappa u^2$ and $\nu(u)=u$.

\begin{figure}[ht]
\centering
\includegraphics[width=15cm,height=5.5cm]{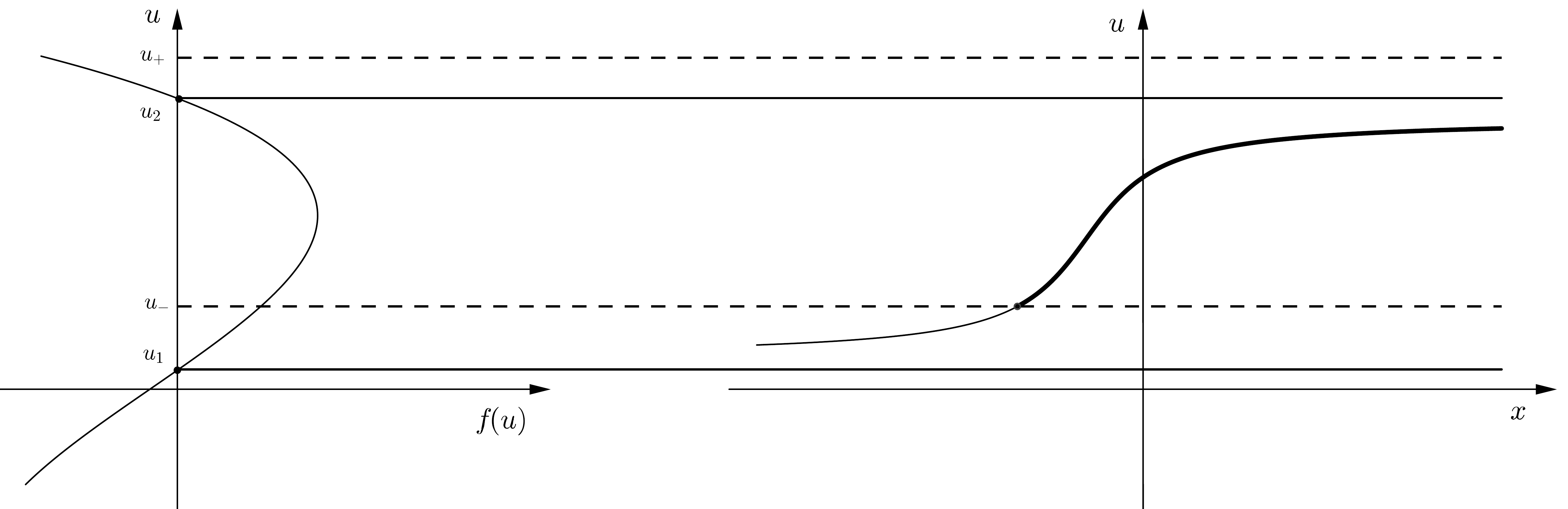}
\caption{\small{Plot of the solutions to $\varepsilon v_*  u_x =f(u)$. The choice for $u_-$ and $u_+$ is such that $u_+>u_2$. In the plane $(x,u)$ we can see that the solution starting from $u(-\ell)=u_-$ can not reach $u_+$, since $u_2$ is an equilibrium solution for the equation.  The same holds if $u_- < u_1$.
}}
\label{primafigura}
\end{figure}

\subsection{The stationary problem}

By taking advantage of the already known properties of the function $f(u)$, we now  study the function $w \mapsto g(w)$. We first notice that the function $f$ is increasing for $u \in [0,u^*)$ and decreasing for $u \in (u^*,+\infty)$, where $u^*$, implicitly defined as
\begin{equation*}
P(u^*)=\alpha-P'(u^*)u^*
\end{equation*}
is such that $f'(u^*)=0$. Moreover, as already stressed in Remark \ref{R1}, if $\alpha$ is such that $f(u^*)>0$, that is
\begin{equation*}
P'(u^*){u^*}^2 >{v_*}^2,
\end{equation*}
then there exist two positive solutions to the equation $f(u)=0$.
Given $\nu(u)>0$, since  $\Phi(u)>0$ and $\Phi'(u)>0$, we have
\begin{equation*}
\Phi^{-1}(w)>0, \quad (\Phi^{-1})'(w)=\frac{1}{\Phi'(u)}>0,
\end{equation*}
proving that $\Phi^{-1}$ is a positive increasing function as well. 

Let us now  consider $g(w)=(f \circ \Phi^{-1})(w)$; we prove the following lemma.

\begin{lemma}\label{nuovolemma}
Let $g(w)=(f \circ \Phi^{-1})(w)$, with $f$ defined in \eqref{definizionef}. For every $v_*>0$ there exist $w_1,w_2>0$ such that $g(w_{1})=g(w_2)=0$. Moreover,  the function $g$ in increasing in the interval $[0,w^*)$, and decreasing in the interval $(w^*,+\infty)$, being $w^*:= \Phi(u^*)$.
 
\end{lemma}

\begin{proof}
Lemma \ref{lemma-1} assures the existence of two positive values $u_1$ and $u_2$ such that $f(u_1)=f(u_2)=0$ and, as a consequence, $w_1$ and $w_2$ has to be defined as
\begin{equation}\label{richiestasuphi}
\Phi^{-1}(w_1)=u_1 \quad \mbox{and} \quad  \Phi^{-1}(w_2)=u_2.
\end{equation}
Since $\Phi^{-1}(0)=0$ and $(\Phi^{-1})'>0$, there exist and they are unique $w_1$ and $w_2$ such that \eqref{richiestasuphi} holds. Hence,  $g(w)$ has exactly two positive zeros for all the choices of $\nu(u)>0$. Furthermore
\begin{equation*}
g'(w)=[f(\Phi^{-1}(w))]'=f'(\Phi^{-1}(w))\cdot(\Phi^{-1})'(w),
\end{equation*}
 so that the sign of $g'$ is univocally determined by the sign of $f'$. Therefore, if ${w^*}$ is such that
$\Phi^{-1}({w^*})=u^*$, then 
\begin{equation*}
g'({w^*})=0, \quad g'(w)>0 \ \textrm{for} \ w \in [0,{w^*}), \quad g'(w)<0 \ \textrm{for} \ w \in ({w^*},+\infty).
\end{equation*}

\end{proof}
We {finally} notice that  condition \eqref{sigmaeq1} for the existence of two positive solutions to the equation $f(u)=0$, also assures that $g(w)$ has to positive zeros. Indeed
\begin{equation*}
g'(w)=f'(\Phi^{-1}(w)) \cdot (\Phi^{-1})'(w)= \frac{f'(\Phi^{-1}(w))}{\Phi'(w)},
\end{equation*}
so that $g'(w)=0$ if and only if $w={w^*}$, where ${w^*}$ is such that $\Phi^{-1}({w^*})=u^*$. Furthermore, \begin{equation*}
g({w^*})>0 \ \Longleftrightarrow \ f(\Phi^{-1}({w^*}))>0 \ \Longleftrightarrow \  f(u^*)>0,
\end{equation*}
which is exactly \eqref{sigmaeq1}.

\begin{example}[The Saint-Venant system with density dependent viscosity]\rm{

When $P(u)= \frac{1}{2}\kappa u^2$, the stationary equation  \eqref{PLgenericostaz}  for $u$ reads
\begin{equation*}
 v_* \varepsilon [\Phi(u)]_x=-\frac{1}{2}\kappa u^3+\alpha u- v_*^2.
\end{equation*}
Let us consider the simplest case  $\nu(u)= Cu^\gamma$, $\gamma>0$, and let us plot the function $g(w)=(f \circ \Phi^{-1})(w)$. We have
\begin{equation*}
\Phi(u)=C\int^u_0 s^{\gamma-1}ds=\frac{C}{\gamma}u^{\gamma} \qquad \mbox{and}  \qquad \Phi^{-1}(w)=\left( \frac{\gamma}{C}u\right)^{\frac{1}{\gamma}},
\end{equation*}
so that 
\begin{equation*}
g(w)=-\frac{1}{2}\kappa\left(\frac{C}{\gamma}\right)^{\frac{3}{\gamma}}w^{\frac{3}{\gamma}}+\left(\frac{C}{\gamma}\right)^{\frac{1}{\gamma}}\alpha \, w^{\frac{1}{\gamma}}-v_*^2.
\end{equation*}

\begin{figure}[ht]
\centering
\includegraphics[width=14cm,height=6cm]{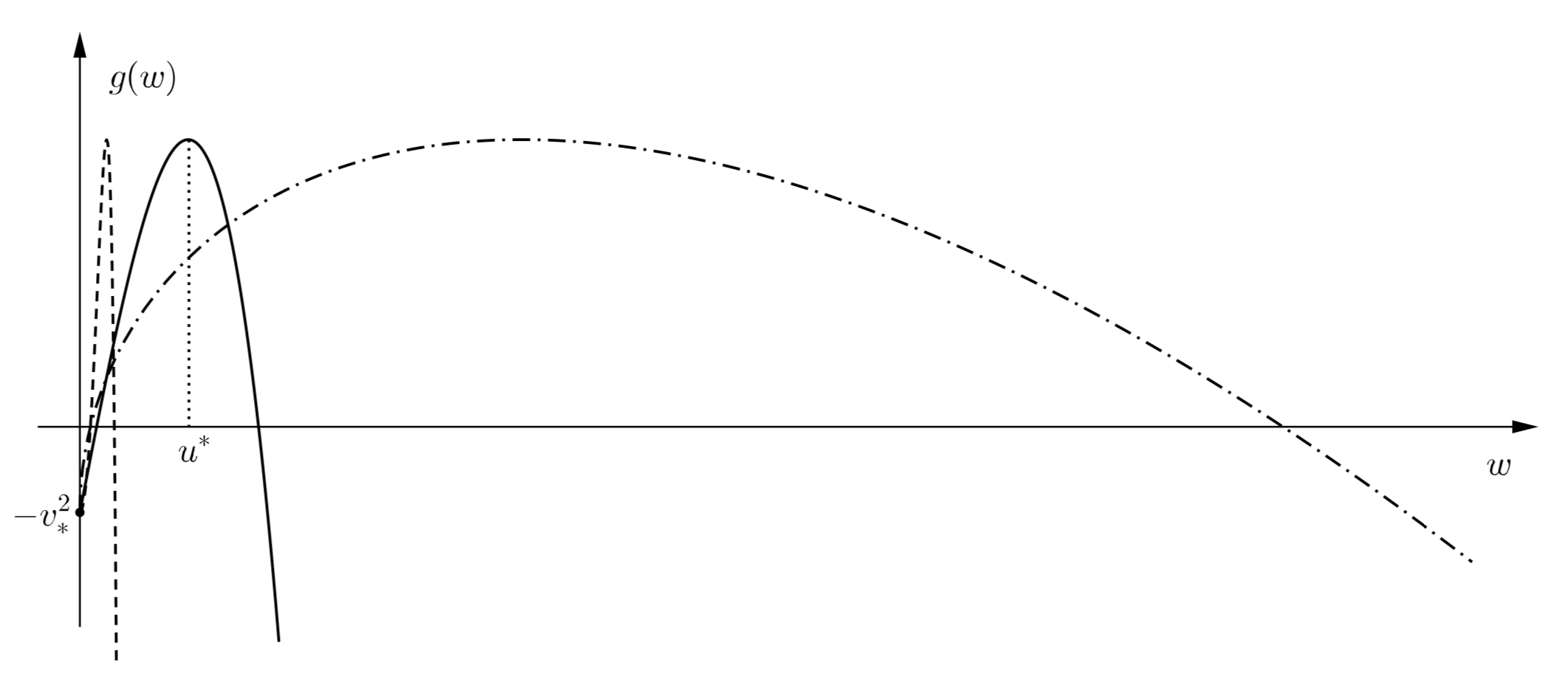}
\caption{\small{Plots of different $g(w)$ with $\kappa=1$, $\alpha=400$ and $ v_*^2=1000$. The dashed line plots $g(w)=-\frac{1}{2}w^6+400w^2-1000$, the dashed-point  line plots $g(w)=-\frac{1}{2}w^{3/2}+400\sqrt{w}-1000$, while the black line plots $f(w)=-\frac{1}{2}w^3-400w-1000$. }}
\label{figesempio}
\end{figure}
Figure \ref{figesempio} shows the plot of $g(w)$ for different choice of  $\nu(u)$, compared with the plot of  $f(w)$ (where $\nu(u)=u$); the dashed line and the dashed point line plot $g(w)$ with $\nu(u)=\frac{\sqrt s}{2}$ and $\nu(u)=2s^2$ respectively. As proved in Lemma \ref{nuovolemma}, we can see that the monotonicity properties of the function $g$ are preserved, as well as the existence of two positive zeros.
}

\end{example}

\subsection{Existence and uniqueness of a positive connection}
Let us go back to the problem of the  existence and uniqueness of the solution to the stationary problem \eqref{PLstaz}. As  already shown, once the boundary conditions for the function $v$ are imposed, problem \eqref{PLstaz} reads
\begin{equation*}
\left\{\begin{aligned}
&v=v_*, \\
&\varepsilon  v_* w_x = g(w), \quad w( \pm \ell)=\Phi(u_\pm)
\end{aligned}\right.
\end{equation*}
where $v_*= v(-\ell)$ and $g(w)=(f \circ \Phi^{-1})(w)$, being $f(u)=-P(u)u+\alpha u- v_*^2$.  

Hence, the equation for the variable $w:=\Phi(u)$ is an equation on the form 
\begin{equation*}
w'=g(w;\alpha), \quad w(\pm \ell)=w_\pm,
\end{equation*}
where $\alpha$ is an integration constant depending on the boundary data. Once the boundary conditions are imposed,
a \textit{positive connection} between $\Phi(u_-)$ and $\Phi(u_+)$ (i.e. a positive solution to $\varepsilon v_*  w_x=g(w)$ connecting $\Phi(u_-)$ and $\Phi(u_+)$) exists only if $$(\Phi(u_-),\Phi(u_+)) \subset (w_1,w_2),$$ being $w_1$ and $w_2$ such that $g(w_1)=g(w_2)=0$.

The following Lemma (to be compared with Lemma \ref{lemma0}) aims at showing some properties of the function $g(w;\alpha)$ as a function of $\alpha$; precisely, we describe how the distance between the two zeroes of the function changes with respect to this parameter.

\begin{lemma}\label{lemmaviscgen1}
Let  $g(w)= (f \circ \Phi^{-1})(w)$ with $f$ defined in \eqref{definizionef}, and let $\alpha$ be such that \eqref{sigmaeq1} holds, so that there exist two positive solutions $w_1<w_2$ to the equation $g(w)=0$. Given $u_\pm >0$, the set $\mathcal A$ defined as
\begin{equation*}
\mathcal A := \{ \alpha >0 : w_1<\Phi(u_-)<\Phi(u_+)<w_2 \},
\end{equation*}
is such that $\mathcal A=[\bar \alpha,+\infty)$, for some $\bar \alpha >0$.

\end{lemma}

\begin{proof}
Since $w_1=w_1(\alpha)$ and $w_2=w_2(\alpha)$,  we want to show that $g(w;\alpha)$ is an increasing function with respect to $\alpha$. Indeed, this would imply that, if there exists a value $\alpha$ such that 
$$w_1<w_-<w_+<w_2,$$
then, for all $\alpha'>\alpha$
$$w_1'<w_-<w_+<w_2',$$
being $w_1'$ and $w_2'$  the  two positive zeros of $g(w;\alpha')$.

Since $g(w;\alpha)=f(\phi^{-1}(w;\alpha))$ and $\Phi^{-1}(w)$ is an increasing function that does not depend on $\alpha$,  $g(w,\alpha)$ is an increasing function in the variable $\alpha$ if so it is for $f(u;\alpha)$. We have
\begin{equation*}
f(u;\alpha)-f(u;\alpha')=(\alpha-\alpha')u,
\end{equation*}
so that, since $u>0$,  $f(u,\alpha')-f(u,\alpha)>0$ when $\alpha'>\alpha$.

Thus, we only need to prove that there exist a value $\bar \alpha$ such that $w_1<\Phi(u_-)<\Phi(u_+)<w_2$. We know that $g(0)=-v_*^2 <0$ and $g'(w)>0$ for all $w \in [0,{w^*})$. Moreover 
\begin{equation*}
g({w^*})=f(\Phi^{-1}({w^*}))=f(u^*)>0,
\end{equation*}
so that $w_1 \in (0,{w^*})$. Furthermore, if we ask for
\begin{equation}\label{condizionecostantilemma2}
g\left( \frac{2v_*^2}{\alpha}\right)=f \left( \Phi^{-1}\left( \frac{2 v_*^2}{\alpha}\right) \right)>0
\end{equation}
we have $w_1<\frac{2v_*}{\alpha}$.  Condition \eqref{condizionecostantilemma2} can be rewritten as
\begin{equation*}
f \left( \Phi^{-1}\left( \frac{2 v_*^2}{\alpha}\right) \right)>f(u_1)=0,
\end{equation*}
that is, since $\Phi^{-1}(w^*)=u_1$
\begin{equation*}
f \left( \Phi^{-1}\left( \frac{2 v_*^2}{\alpha}\right) \right)>f\left( \Phi^{-1}\left( w_1\right)  \right).
\end{equation*}
Since $f$ and $\Phi^{-1}$ are increasing function in the interval $[0,u^*)$ and $[0,{w^*})$ respectively, we obtain the following condition for the constant $\alpha$ 
$$2 v_*^2/\alpha >w_1.$$
If this condition holds, then 
\begin{equation*}
0<w_1<\frac{2 v_*}{\alpha},
\end{equation*}
showing that $w_1 \to 0$ as $\alpha \to +\infty$. On the other hand we know that $u_2>u^*$ where $u^*$ is such that $f(u^*)=\max_\R f$. Hence
\begin{equation*}
\Phi^{-1}(w_2)>\Phi^{-1}({w^*}) \ \ \ \Rightarrow \ \ \  w_2 > \Phi(u^*).
\end{equation*}
Since $u^* \to + \infty$ as $\alpha \to +\infty$, and since $\Phi$ is an increasing and continuous function, we know that $\Phi(u^*) \to \Phi(+\infty)=+\infty$ as $\alpha \to + \infty$, implying $w_2 \to +\infty$ as $\alpha \to +\infty$.
\vskip0.2cm
We  have thus proved that, if we choose $\bar \alpha$ large enough, then $(\Phi(u_-),\Phi(u_+)) \subset (w_1,w_2)$ for every choice of $u_\pm >0$.
More precisely, $\bar \alpha$ is chosen in such a way that
\begin{equation*}
\bar \alpha>\max \{ \alpha^*, \alpha^{**} \},
\end{equation*}
where $\alpha^*$ and $\alpha^{**}$ are such that either $g(\Phi(u_-),\alpha^*)=0$ or $g(\Phi(u_+),\alpha^{**})=0$.

\end{proof}

\begin{ans}\rm{

We  define the region $\Sigma$ of admissible values $\alpha$ as the set of all the values $\alpha$ such that there exists two positive solutions to the equation $g(w)=0$ and Lemma \ref{lemmaviscgen1} holds. In the plane $\{ v^*,\alpha\}$, $\Sigma$ is determined  by the equations
\begin{equation*}
v_*^2 <P'(u^*){u^*}^2, \quad g(\Phi(u_\pm))>0, \quad \alpha< \frac{2v_*}{w_1}.
\end{equation*}
We recall that $u^*$ is such that $f'(u^*)=0$ and $v_*=v(\pm \ell)$.
}
\end{ans}

\begin{proposition}\label{propviscgen1}

The region $\Sigma$ is the {\bf epigraph} of an increasing function $h: \R \to \R$, i.e.
\begin{equation*}
\Sigma := {\bf epi}(h)= \{ (v_*,\alpha) : v_* \in \R, \alpha \in \R, \alpha \geq h(x) \} \subset \R \times \R.
\end{equation*}
 
 \end{proposition}

 \begin{proof}
 
Setting $\varphi(\alpha):= \sqrt{P'(u^*)}u^*$, we have
\begin{equation*}
\lim_{\alpha \to 0}\varphi(\alpha)=0, \quad \lim_{\alpha \to +\infty} \varphi(\alpha)=+\infty, \quad \varphi'(\alpha)>0,
\end{equation*}
meaning that $v_*=\varphi(\alpha)$ is an increasing function in the plane $(v_*,\alpha)$. Moreover, the condition $g(\Phi(u_\pm))>0$ is equivalent to
 \begin{equation*}
g(\Phi(u_\pm))= (f\circ \Phi^{-1})(\Phi(u_\pm))=f(u_\pm)>0,
\end{equation*}
and we get
 \begin{equation*}
 \alpha>\frac{1}{u_-}v_*^2+P(u_-), \quad  \alpha>\frac{1}{u_+}v_*^2+P(u_+)
 \end{equation*}
whose equality defines two parabolas. Finally, the function $\Psi(\alpha)=\frac{2v_*}{w_1(\alpha)}$ is such that
\begin{equation*}
\lim_{\alpha \to +\infty}\Psi(\alpha)=+\infty \quad \mbox{and} \quad \Psi'(\alpha)=-\frac{2v^*}{w_1^2} w'_1 >0,
\end{equation*}
since $w_1(\alpha)$ is a decreasing function. Hence  $\frac{dh}{dv_*}>0$, since $h$  is obtained by matching increasing functions.
  
 \end{proof}

We now prove the existence of a \textit{$2\ell$-connection}, i.e. we prove the existence of a solution to 
$$\varepsilon v_*  w_x=g(w), \quad w(\pm \ell)=\Phi(u_\pm),$$
satisfying
\begin{equation*}
2\ell=\varepsilon  v_* \int_{\Phi(u_-)}^{\Phi(u_+)} \frac{dw}{(f\circ \Phi^{-1})(w)} := G(\alpha).
\end{equation*}
We first notice that $G\big|_{\partial \Sigma}=+\infty$. From the study of $G(\alpha)$, we can prove that there  exists a unique value $\alpha^*$ such that $G(\alpha^*)=2\ell$. Indeed, we can easily see that
\begin{equation*}
\begin{aligned}
\lim_{\alpha \to + \infty} G(\alpha)=0, \quad \lim_{\alpha \to \bar \alpha} G(\alpha)= +\infty \quad  \mbox{and} \quad \frac{dG}{d\alpha} <0.
\end{aligned}
\end{equation*}
for  $\bar \alpha \in \partial\Sigma $ and for all $\alpha >0$.

\vskip1cm
 We are finally able to prove Theorem \ref{teointro}, which we recall here for completeness.
 
\begin{theorem}\label{teo2}
Given $\ell>0$ and $u_\pm, v_->0$, let us consider the following problem
\begin{equation}
 \left\{\begin{aligned}\label{PLgenericoTeo}
 &u_t +  v_x=0 &\qquad &x \in (-\ell,\ell), \ t \geq 0 \\
&  v_t+\left\{ \frac{v^2}{u}+ P(u)-\varepsilon \, \nu(u) \left( \frac{v}{u}\right)_x\right\}_x=0   \\
& u(\pm \ell,t)= u_\pm, \quad v(- \ell, t)=v_-, &\qquad &t \geq 0 \\
& u(x,0)=u_0(x), \quad v(x,0)=v_0(x) &\qquad &x \in (-\ell,\ell).
\end{aligned}\right.
\end{equation}
where $P(u)$ and $\nu(u)$ satisfy hypotesis {\bf{H1}}. If  $u_\pm$ and $v_*$ verify
$${\bf H2.} \quad v_*^2 \, (u_+-u_-)= u_- \,u_+ \, [P(u_+)-P(u_-)] \quad {\rm and} \quad  \left[\!\!\left[ \frac{v^3}{2u^2}+v f'(u)  \right]\!\!\right] \leq 0,
$$
being $f(u):= u \, \int P(z)/z^2 \, dz$ and  $v^* = v_->0$, then there exists a unique stationary solution $(\bar u(x), \bar v(x))$ to \eqref{PLgenericoTeo}.

\end{theorem}

\begin{proof}
As already mentioned, the second component of the steady state is univocally determined once the boundary conditions are imposed, that is $\bar v(x)\equiv v_*$.

Going further, once $v_*$ is given, Lemma \ref{lemmaviscgen1} assures that, for any choice of $u_\pm$, there exists at least a value $\alpha$ such that $(v_*,\alpha) \in \Sigma$ and $w_1<\Phi(u_-)<\Phi(u_+)<w_2$, so that there exists a \textit{positive connection} satisfying the boundary conditions. Moreover, from the study of the function $G(\alpha)$, we know that there exists a unique value $\alpha^*$ such that $(v_*,\alpha^*) \in \Sigma$ and $G(\alpha^*)=2\ell$, so that there exists a unique \textit{positive connection} $\bar w(x)$ between $\Phi(u_-)$ and $\Phi(u_+)$ of \textquotedblleft length" $2\ell$. Since $\Phi$ is invertible, $\bar u(x):=\Phi^{-1}(\bar w)$ is the unique positive connection between $u_-$ and $u_+$.

\end{proof}

\subsection{The Saint-Venant system}{\rm

An interesting case where we can explicitly develop computations is the Saint-Venant system, already studied in \cite{Str14}; here the term of pressure $P(u)$ is given by the quadratic formula $P(u)=\frac{1}{2}\kappa u^2$, $\kappa>0$ and the viscosity $\nu(u)=u$.

In this case, stationary solutions solve
\begin{equation*}
v= v_* \qquad \mbox{and} \qquad \varepsilon  v_*   u_x=-\frac{1}{2}\kappa u^3+\alpha u- v_*^2:=f(u),
\end{equation*}
where, as usual, $v_*=v_-$. The condition \eqref{sigmaeq1}
for the existence of two positive solution $u_1$ and $u_2$ enlightened in Remark \ref{R1} reads $\alpha^3 > 27/8\,\kappa\, v_*^4$ (which is exactly the Cardano condition for the existence of three real solutions to third order equations  in the form $u^3+pu+q=0$).
Moreover, since $f(0)=-v_*^2$ and $\alpha>0$, we can explicitly show that $u_0<0<u_1<u_2$, where $u_0$ is the third (negative) root of the equation $f(u)=0$.

Figure \ref{secondafigura} plots the function $f(u)$ for different choices of the constant $\alpha$. The picture explicitly shows how the first positive zero $u_1$ remains close to zero while  $u_2$ becomes bigger as $\alpha \to +\infty$. Figure \ref{secondafigura} also shows that the interval $(u_-,u_+)$ is included or not inside $(u_1,u_2)$, depending on the choice of $\alpha$. In Figure \ref{terzafigura} we plot on the phase plane the solution to the equation $\varepsilon v_* \, u_x =f(u)$, which is known to exist once the boundary   values are chosen so that $(u_-,u_+) \subset (u_1,u_2)$. Moreover $u_1$ and $u_2$, being zeros of the function $f(u)$, are equilibria for the equation.

\begin{figure}[ht]
\centering
\includegraphics[width=12cm,height=6cm]{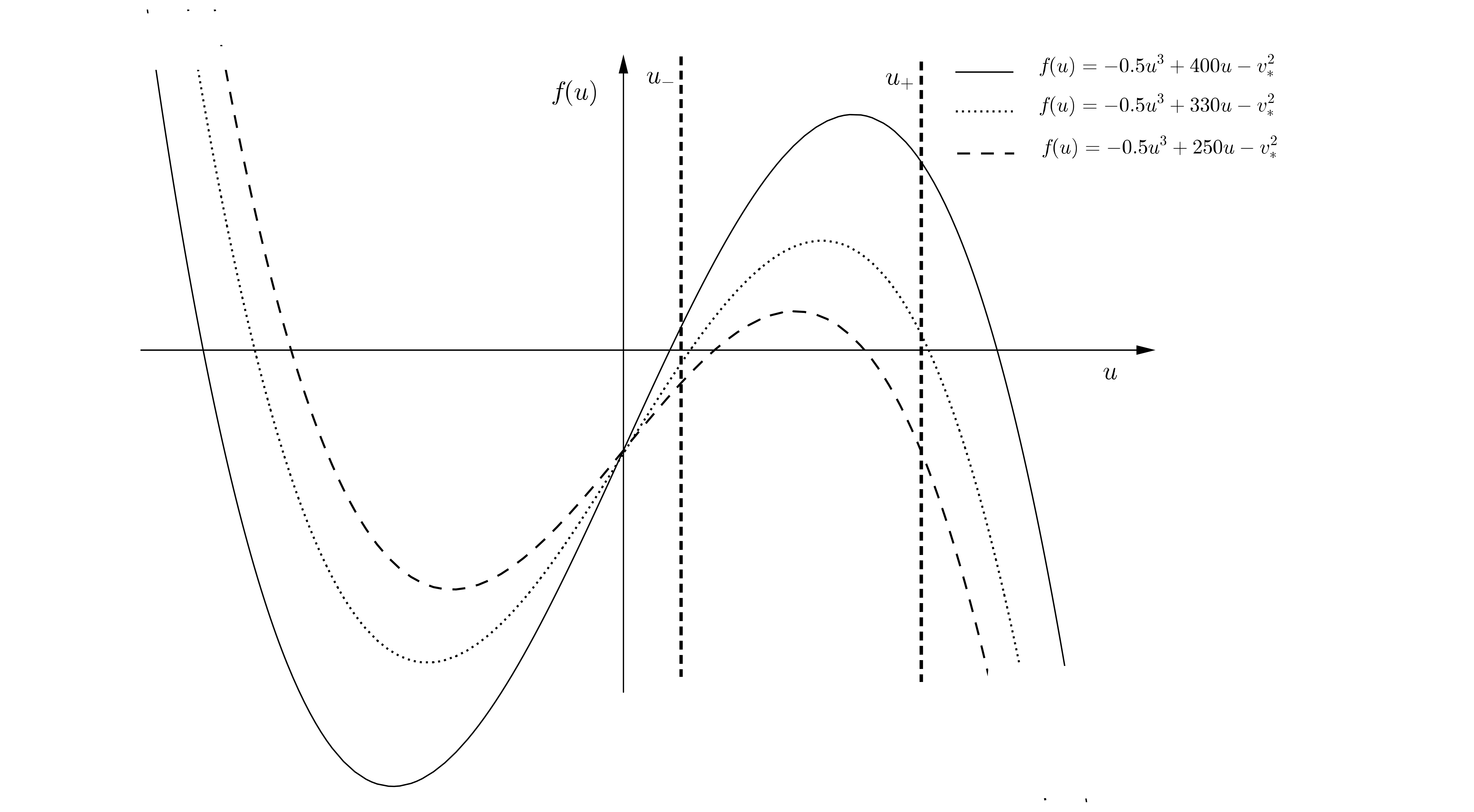}
\caption{\small{Plot of $f(u)=-\frac{1}{2}\kappa u^3+\alpha \,u- v_*^2$ for fixed $ v_*$ and multiple choices of $\alpha$.}}
\label{secondafigura}
\end{figure}

\begin{figure}[ht]
\centering
\includegraphics[width=14cm,height=6cm]{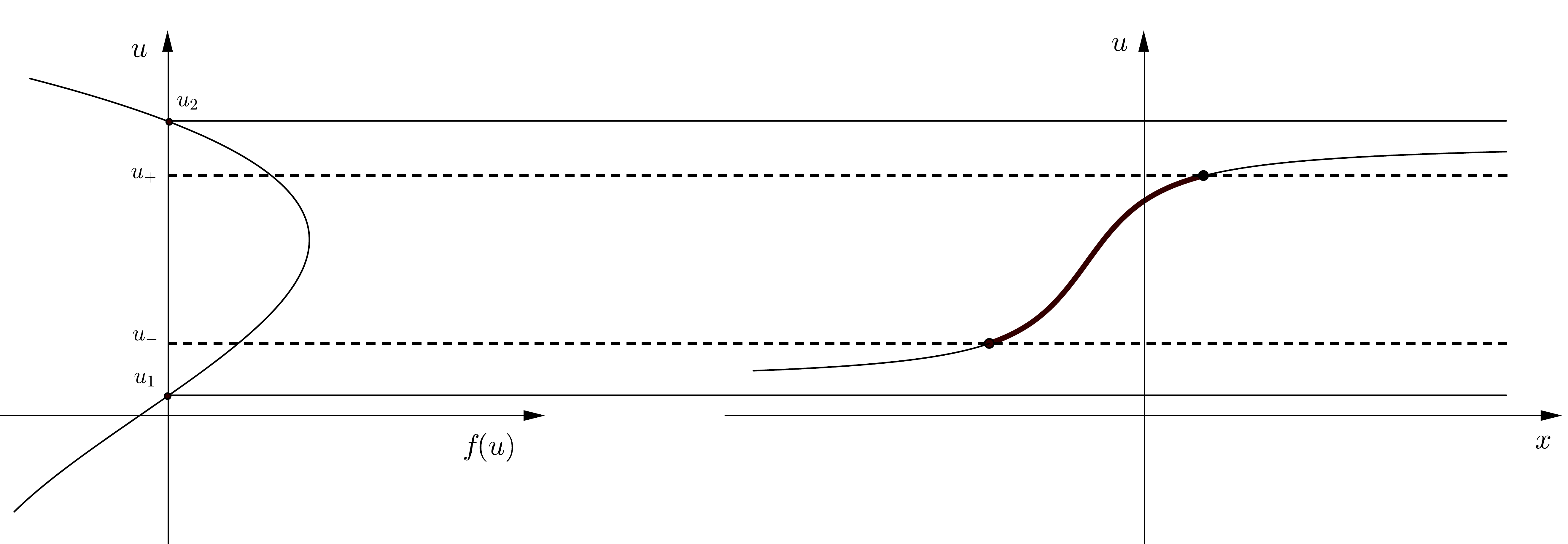}
\caption{\small{Plot of the solutions to $\varepsilon v_* u_x =f(u)$. In this case $(u_-,u_+)\subset(u_1,u_2)$, so that there exists a positive connection between $u_-$ and $u_+$. }}
\label{terzafigura}
\end{figure}

}

\section{Stability properties of the steady state}

In this Section we  study the stability properties of the unique steady state $(\bar u,\bar v)$ to the Navier-Stokes equations
\begin{equation}
 \left\{\begin{aligned}\label{NS}
 & u_t +  v_x=0 &\qquad &x \in (-\ell,\ell), \ t \geq 0 \\
& v_t+\left\{ \frac{v^2}{u}+ P(u)-\varepsilon \, \nu(u) \left( \frac{v}{u}\right)_x\right\}_x=0   \\
& u(\pm \ell,t)= u_\pm, \quad v(- \ell, t)=v_-, &\qquad &t \geq 0, \\
& u(x,0)=u_0(x), \quad v(x,0)=v_0(x) &\qquad &x \in (-\ell,\ell),
\end{aligned}\right.
\end{equation}
which is known to exist and to be unique thanks to Theorem \ref{teo2}. As stated in the Introduction, the key tool we are going to use is  the construction of a Lyapunov functional for \eqref{NS}, and the subsequent use of a Lyapunov type stability theorem; a similar strategy has been already used in \cite{KNZ}, where the authors prove asymptotic stability for the steady state of the Navier-Stokes system in the half line. We here prove stability in the sense of Definition \ref{defintro}, and our goal is to prove Theorem \ref{teointro2}, providing an estimate on the $L^2$-norm of the difference $(u,v)-(\bar u,\bar v)$, being $(u,v)$ the solution to \eqref{NS}.

\vskip0.2cm
Before going through the explicit construction of the Lyapunov functional and the computation of its time derivative, let us recall the additional hypothesis needed, as well as some useful observations on the behavior of the derivative of the solution $(u,v)$ at the boundary.

\subsubsection*{\bf Hypotheses:} As stated in the Introduction, in order to prove the stability of the steady state  we need to require a smallness assumption on the value of the density $u$ and its derivative at the boundary; precisely
we require the boundary data $u(\pm \ell,t)=u_\pm$ and $u_x(\pm\ell,t)$ to satisfy the following  condition:
\vskip0.2cm
{\bf H3.} There exist  positive constants $ \delta_1, \delta_2$ such that 
$$|u_+-u_-| < \delta_1 \qquad \mbox{and} \qquad |u_x(\ell,t)-u_x(-\ell,t)| < \delta_2,$$
for all $t \geq 0$.
\vskip0.2cm
\noindent As already remarked,  similar requests as the one in {\bf H3}, providing a smallness condition on the density $u$, have already  been stated in \cite{MN, ZZZ} in order to have stability of the steady state for  a Navier-Stokes system in the half line.

\subsubsection*{\bf Behavior at the boundary of the derivative of the time dependent solution:} Let us notice that from the first equation in \eqref{NS} and since $u$ satisfies Dirichlet boundary conditions, it follows
\begin{equation}\label{neuman}
v_x(\pm \ell,t)=0, \qquad \forall \, t \geq 0, 
\end{equation}
that is, $v$ satisfies Neumann boundary conditions. Such property will be used when computing the sign of the time derivative of the Lyapunov functional along the solution, and it
 can be also observed  by numerically computing the solution $(u,v)$ to \eqref{NS}, as done in  Figure \ref{ipoL}. 
 
 The initial datum for $v$ is the constant function $v_0=v_*$; the time dependent solution has always zero derivative at the boundary and, after developing into a non-constant solution, we can see its convergence towards the stable steady state $\bar v=v_*$.

As concerning $u$,  we choose as initial datum the increasing function  $u_0=0.25\tanh(20x)+0.75$; in this case $u_+=1$ and $u_-=0.5$, so that the first request in assumption ${\bf H3}$ is satisfied for all $\delta_1>0.5$. Indeed, as already remarked,  if the initial datum for the density  satisfies $u_- \leq u_0(x) \leq u_+$ for all $x \in [-\ell, \ell]$, then this property is preserved by the dynamics  (see Lemma \ref{A2} and the subsequent Remark). Also, one can see how the derivative at the boundary remain bounded, so that there exists $\delta_2>0$ such that the second assumption in ${\bf H3}$ is satisfied. 
Finally,
we notice that the solution itself remain increasing  for all $x \in [-\ell,\ell]$ and for all $t > 0$.
\begin{figure}[ht]
\centering
\includegraphics[width=7.3cm,height=5.7cm]{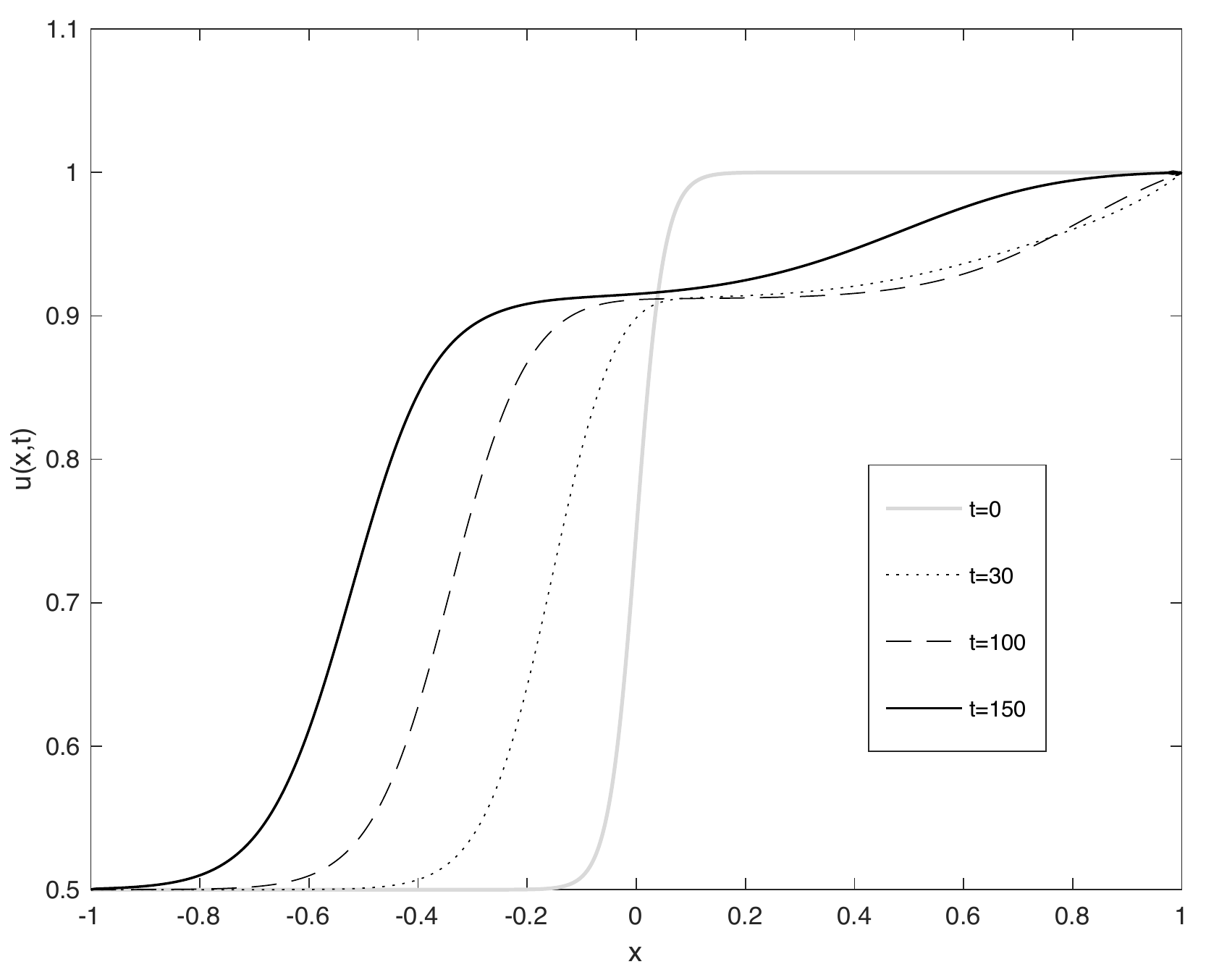}
\includegraphics[width=7.3cm,height=5.7cm]{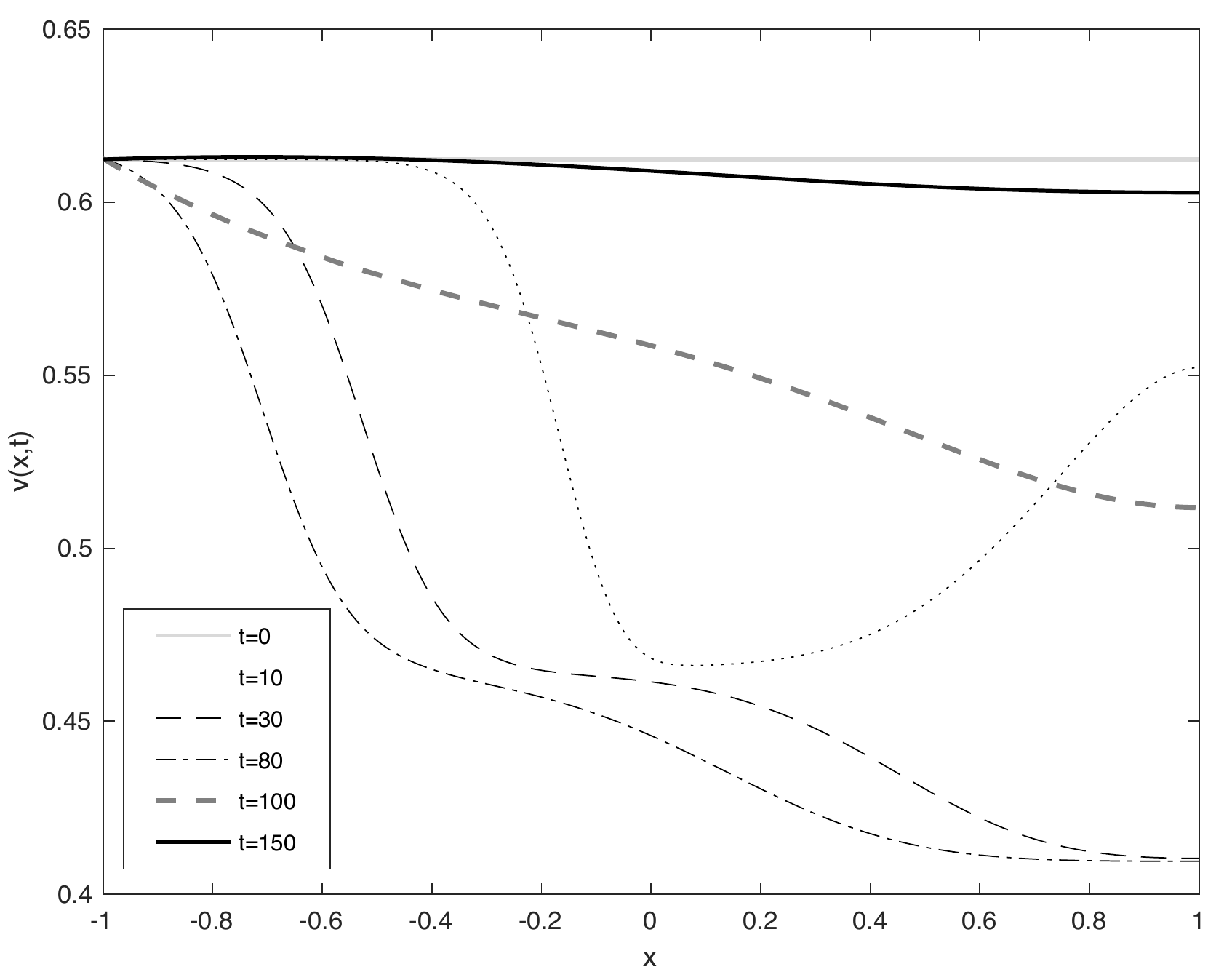}
\caption{\small{The dynamics of the solution to \eqref{EqTeo} in the interval $[-1,1]$ with  $\e=0.1$, $P(u)=u^2/2$ and $\nu(u)=u$. The initial datum $u_0$ is an increasing function connecting $u_-=0.5$ and $u_+=1$, while $v_0$ is such that $v_0(-\ell)=v_*$, with $v_*$ satisfying the equality in assumption {\bf H2}.
}}
\label{ipoL}
\end{figure}

\subsubsection*{\bf Notations:} Throughout this section, we shall write $A \lesssim B$ if there exists a positive constant $C$  such that 
$$A \leq C \, B.$$ 
Also, for the sake of shortness, we will omit the dependence of $(u,v$) from the variables $(x,t)$.
Finally, given two functional spaces $X$ and $Y$ and  a function $f$  of the two variables $(x,t) \in [-\ell,\ell] \times [0,T]$  such that $f( x,\cdot) \in X $ and $f(\cdot, t) \in Y$, we will denote with 
$$|f|_{{}_{XY}} := |f|_{{}_{X([0,T], Y(I))}}.$$

\noindent We are now ready to define the Lyapunov functional.

\subsection{Construction of the Lyapunov functional} 
 
System \eqref{NS} admits a mathematical entropy which is also a physical energy
\begin{equation*}
\mathcal E(u,v):=\frac{v^2}{2u}+u \, \phi(u), \quad \phi( u ) =  \int^{u}_0 \frac{P ( z )}{ z^2} \, dz,
\end{equation*}
where we recall that $u \, \phi(u) + P(u)=u \, f'(u)$, being
\begin{equation*}
 f(u)= u \,\int^u_0 \frac{P(z)}{z^2} \, dz.
\end{equation*}
In the present setting, since $\e>0$, the entropy flux is given by 
\begin{equation*}
 \mathcal Q(u,v):=\frac{v}{u}\left[ \frac{v^2}{2u}+u \, \phi(u) +P(u)-\e \frac{v}{u}\left( \nu(u)\left(\frac{v}{u}\right)_x\right)\right],
\end{equation*}
and the energy equality thus becomes
\begin{equation}\label{energy2}
\begin{aligned}
\frac{1}{2}\left(\frac{v^2}{2u}+u \, \phi(u)\right)_t+ &\,\left[ \frac{v^3}{2u^2}+v \,\phi(u)+P(u)-\e \frac{v^2}{u^2}\left( \nu(u)\left(\frac{v}{u}\right)_x\right)\right]_x= \\
&=-\varepsilon \nu(u)\left[\left(\frac{v}{u}\right)_x\right]^2.
\end{aligned}
\end{equation}
Inspired by the results in \cite{KNZ}, we  introduce the new energy form (usually referred as {\it modulated energy})
\begin{equation*}
\mathcal E_L(u,v, \bar u, \bar v) := \frac{(v-\bar v)^2}{2u} + u \, \psi (u, \bar u), \quad \psi (u, \bar u)= \int_{\bar u}^u \frac{P(z)-P(\bar u)}{z^2} \, dz,
\end{equation*}
and we claim that a good candidate to be a Lyapunov functional for the system is
\begin{equation}\label{L}
L(u,v, \bar u, \bar v) :=  \int_{-\ell}^\ell \mathcal E_L(u,v, \bar u, \bar v) \, dx.
\end{equation}
Indeed, \eqref{L} is of course null as computed on the steady state $(u,v)=(\bar u, \bar v)$ and positive defined since $P'(u)>0$.

\subsection{Computation of the time derivative of $L$ along the solutions.} We want to  compute the time derivative of $\eqref{L}$, showing that it is negative along the solutions  to \eqref{NS}. We have 
\begin{equation*}
\begin{aligned}
L(u,v,\bar u,\bar v) &=  \int_{-\ell}^\ell \left\{\frac{v^2}{2u} + \frac{\bar v^2}{2u}-\frac{v \bar  v}{u} + u \, \psi(u,\bar u) \right\} \,dx,\\
\end{aligned}
\end{equation*}
and we observe that
\begin{equation*}
\psi(u,\bar u) = \int_{\bar u}^u \frac{P(z)-P(\bar u)}{z^2} \, dz = \int_{0}^u \frac{P(z)}{z^2} \, dz - \int_{0}^{\bar u} \frac{P(z)}{z^2} \, dz -\int_{\bar u}^u \frac{P(\bar u)}{z^2} \, dz. 
\end{equation*}
In particular, $L$ can be rewritten as
\begin{equation}\label{funzlya}
L(u,v, \bar u, \bar v) =  \int_{-\ell}^\ell \left\{\frac{v^2}{2u}+ u \,\phi(u) + \frac{\bar v^2}{2u}-\frac{v\bar  v}{u} - u \int_0^{\bar u} \frac{P(z)}{z^2} \, dz+ P(\bar u) \frac{\bar u-u}{\bar u} \right\} \,dx.
\end{equation}
The following proposition holds.

\begin{proposition}\label{dtLnegative}
Let assumption {\bf H3} be satisfied; then, for $\delta_1$ and $\delta_2$ sufficiently small,  we have 
\begin{equation*}
\frac{d}{dt} L(u,v,\bar u,\bar v) \leq 0,
\end{equation*}
for all $(u,v)$ solutions to \eqref{NS}, being $(\bar u, \bar v)$  the unique steady state of \eqref{NS} and $L$ as in \eqref{L}. 
\end{proposition}
\begin{remark}{\rm
Throughout the proof, we will extensively make use of the positivity, for all $t\geq0$, of the function $u(x,t)$ and its space derivative in the interval $[-\ell,\ell]$.
Also, we will use the results of Appendix A, providing the existence of a solution $(u,v) \in  L^\infty([0,T], H^1(I))\times  L^\infty([0,T], H^1(I))$; in particular, we will use the fact that the quantities $|(u,v)|_{{}_{L^\infty H^1}}$ and $|(u,v)|_{{}_{L^\infty L^\infty}}$ are finite.
}
\end{remark}

\begin{proof}
\vskip0.3cm
\noindent By taking advantage of the energy equality \eqref{energy2}, from \eqref{funzlya} we get
\begin{equation}\label{derivataL}
\begin{aligned}
\frac{d}{dt}L&
= \int_{-\ell}^\ell\left\{-\left[ \frac{v^3}{2u^2}+v\, f'(u)-\e \frac{v^2}{u^2}\left( \nu(u)\left(\frac{v}{u}\right)_x\right)\right]_x-\varepsilon \nu(u)\left[\left(\frac{v}{u}\right)_x\right]^2 \right\} \, dx \\
& \quad + \frac{d}{d t}\int_{-\ell}^\ell  \left\{\left[  \frac{\bar v^2}{2u}-\frac{v \bar v}{u}- u \int_0^{\bar u} \frac{P(z)}{z^2} \, dz+P(\bar u) \frac{\bar u-u}{\bar u}\right] \right\} \, dx \\
&=
- \left[\frac{v}{u}\left(\frac{v^2}{2u}+u\, f'(u)\right)-\e \frac{v^2}{u^2}\left( \nu(u)\left(\frac{v}{u}\right)_x\right)\right]\Big|^{\ell}_{-\ell} -\varepsilon \int_{-\ell}^\ell \nu(u)\left( \left( \frac{v}{u}\right)_x\right)^2 \, dx\\
&\quad + \frac{d}{d t}\int_{-\ell}^\ell  \left\{\left[  \frac{\bar v^2}{2u}-\frac{v \bar v}{u}- u \int_0^{\bar u} \frac{P(z)}{z^2} \, dz+P(\bar u) \frac{\bar u-u}{\bar u}\right] \right\} \, dx\\
&={\bf A_\e}(u,v,u_x,v_x)\big|^{x=\ell}_{x=-\ell} +{\bf B_\e}(u,v,u_x,v_x)  \\
&\quad + \frac{d}{dt} \int_{-\ell}^{\ell}\left\{\left[  \frac{\bar v^2}{2u}-\frac{v \bar v}{u}- u \int_0^{\bar u} \frac{P(z)}{z^2} \, dz+P(\bar u) \frac{\bar u-u}{\bar u}\right] \right\} \, dx,
\end{aligned}
\end{equation}
with notation
\begin{equation*}
\begin{aligned}
&{\bf A_\e}(u,v,u_x,v_x):= - \left[\frac{v}{u}\left(\frac{v^2}{2u}+u\, f'(u)\right)-\e \frac{v^2}{u^2}\left( \nu(u)\left(\frac{v}{u}\right)_x\right)\right], \\
&{\bf B_\e}(u,v,u_x,v_x) :=-\varepsilon \int_{-\ell}^\ell \nu(u)\left( \left( \frac{v}{u}\right)_x\right)^2 \, dx.
\end{aligned}
\end{equation*}
The  term {$\bf B_\e$} is negative. In order to check the sign of {$\bf A_\e$}, 
we preliminary recall that
\begin{equation*}
u\, \phi(u)+P(u)= f(u)+P(u)= u \, f'(u),
\end{equation*}
with notations introduced before
\begin{equation*}
 f( u ) =  u \, \int^{u}_0 \frac{P ( z )}{ z^2} \, dz \quad \mbox{and} \quad  \phi( u ) =  \int^{u}_0 \frac{P ( z )}{ z^2} \, dz.
\end{equation*}
We have
\begin{equation}\label{dtl1}
\begin{aligned}
{\bf A_\e} \, \Big|_{x=\ell}&=-\left[\frac{v_*}{u_+}\left( \frac{v_*^2}{2u_+}+u_+\, f'(u_+)\right)-\varepsilon\frac{v_*^2}{u^2_+}\left(\nu(u_+) \left[\frac{v_x(\ell,t)u_+-v_*u_x(\ell,t)}{u_+^2}\right]\right)\right], \\
&=-\left[\frac{v_*}{u_+}\left( \frac{v_*^2}{2u_+}+u_+\, f'(u_+)\right)+\varepsilon\frac{v_*^3}{u^4_+}\nu(u_+) u_x(\ell,t)\right],
\end{aligned}
\end{equation}
where in the last equality  we used \eqref{neuman}.
By using the same arguments for $x=-\ell$, we obtain
\begin{equation}\label{dtl4}
\begin{aligned}
-{\bf A_\e} \, \Big|_{x=-\ell}=\left[\frac{v_*}{u_-}\left( \frac{v_*^2}{2u_-}+u_-\, f'(u_-)\right)-\varepsilon\frac{v_*^3}{u^4_-}\nu(u_-) u_x(-\ell,t)\right],
\end{aligned}
\end{equation}
so that, summing \eqref{dtl1} and \eqref{dtl4} and recalling that $u_\pm \, f'(u_\pm)= P(u_\pm)+f(u_\pm)$, we end up with
\begin{equation*}
\begin{aligned}
{\bf A_\e}\Big|^{\ell}_{-\ell}&=
\frac{v_*^3}{2}\left( \frac{1}{u^2_-}-\frac{1}{u^2_+}\right) + v_*  \left[\frac{P(u_-)}{u_-}-\frac{P(u_+)}{u_+} \right] + v_*\left[\frac{f(u_-)}{u_-}-\frac{f(u_+)}{u_+} \right] \\
&\quad +  \e v_*^3\left[ \frac{\nu(u_-)}{u^4_-}u_x(-\ell,t)-  \frac{\nu(u_+)}{u^4_+}u_x(\ell,t)\right] .
\end{aligned}
\end{equation*}
 Since ${v_*^3}>0$, the first term in the above sum can be bounded via the difference $u_+ -u_-$ (and hence by $\delta_1$, see assumption {\bf H3});  
concerning the second term we have
\begin{equation*}
\begin{aligned}
P(u_-)u_+ - P(u_+)u_- &= P(u_-)u_+-P(u_-)u_-+ P(u_-)u_- - P(u_+)u_- \\
&= P(u_-)(u_+ - u_-) + u_-(P(u_-)-P(u_+)), 
\end{aligned}
\end{equation*}
which again can be bounded from above via the difference $u_+-u_-$, since $P(u_-) < P(u_+)$. For the third term, we recall that, by definition
\begin{equation*}
\frac{f(u_-)}{u_-}-\frac{f(u_+)}{u_+}= \int_0^{u_-} \frac{P(z)}{z^2} \, dz- \int_0^{u_+} \frac{P(z)}{z^2} \, dz,
\end{equation*}
and this difference  is negative since $u_- < u_+$.  Finally, for the last term, we first observe that
 \begin{equation}\label{uselater}
 \begin{aligned}
   \frac{\nu(u_-)}{u^4_-} u_x(-\ell,t) &-  \frac{\nu(u_+)}{u^4_+} u_x(\ell,t) = \\
   &=  \frac{\nu(u_-)}{u^4_-} u_x(-\ell,t)- \frac{\nu(u_-)}{u^4_-} u_x(\ell,t)+ \frac{\nu(u_-)}{u^4_-} u_x(\ell,t) -   \frac{\nu(u_+)}{u^4_+} u_x(\ell,t)  \\
   &= \frac{\nu(u_-)}{u^4_-} ( u_x(-\ell,t)- u_x(\ell,t))+ u_x(\ell,t) \left[ \frac{\nu(u_-)}{u^4_-}-   \frac{\nu(u_+)}{u^4_+}\right],
  \end{aligned}
  \end{equation}
  and, by assumption {\bf H3}, these two quantities can be bounded from above by $\delta_2$ and $\delta_1$ respectively.

\vskip0.2cm
In order to compute the sign of the time derivative of $L(u,v,\bar u,\bar v)$ given in \eqref{derivataL}, we are thus left with evaluating the sign of
\begin{equation}\label{rimasti}
\begin{aligned}
\frac{d}{dt}\int_{-\ell}^\ell &  \left[\frac{v_*^2}{2u}-\frac{v v_*}{u}- u \int_0^{\bar u} \frac{P(z)}{z^2} \, dz-P(\bar u) \frac{\bar u-u}{\bar u}\right] \, dx : =  \\
\quad & \quad  \quad \left[ A(u,v,v_*)+ B(u,v,v_*)+ C(u,\bar u)+ D(u,\bar u) \right].
\end{aligned}
\end{equation}
\subsubsection*{Computation of the sign of $A(u,v,v_*)$ }   We have, by integration by parts 
\begin{equation}\label{miserve0}
\begin{aligned}
A(u,v,v_*)&=-\frac{1}{2}\int_{-\ell}^\ell v_*^2\frac{ u_t}{u^2} \, dx =\frac{v_*^2}{2}\int_{-\ell}^\ell \frac{ v_x}{u^2} \, dx \\
&=\frac{v_*^2}{2} \left[  \frac{v}{u^2} \Big|^{\ell}_{-\ell}+ 2\int_{-\ell}^\ell \frac{ u_x \, v}{u^3} \, dx\right] \\
& \lesssim \frac{v_*^2}{2} \left[ v_*\left(\frac{1}{u_+^2}-\frac{1}{u_-^2}\right)+ 2\frac{ {|v|_{{}_{L^\infty L^\infty}}}}{u_-^3} (u_+-u_-)\right].
 \end{aligned}
\end{equation}
where we used (recall that $u_x >0$)
\begin{equation*}
\int_{-\ell}^\ell \frac{ u_x \, v}{u^3} \, dx \leq \int_{-\ell}^\ell \frac{| u_x| \, |v|}{|u|^3} \,dx \leq \frac{|v|_{{}_{L^\infty L^\infty}}}{u_-^3}\int_{-\ell}^\ell | u_x| \, dx= \frac{|v|_{{}_{L^\infty L^\infty}}}{u_-^3}\int_{-\ell}^\ell  u_x \, dx \lesssim u_+-u_-.
\end{equation*}
The first term on the right hand side of the last line in  \eqref{miserve0} is negative, while the last one can be bounded via dei difference $u_+-u_-$.
\subsubsection*{Computation of the sign of $B(u,v,v_*)$ } Such term  needs  more care; we have
\begin{equation}\label{tantitermini}
\begin{aligned}
-v_*\int_{-\ell}^\ell  \left( \frac{v}{u}\right)_t \, dx &= v_* \int_{-\ell}^\ell \frac{v  u_t- v_t u}{u^2} \, dx \\
&=-v_* \int_{-\ell}^\ell \frac{v  v_x}{u^2} \, dx \\
&\quad +v_* \int_{-\ell}^\ell \frac{1}{u} \left[ \frac{v^2}{u}+P(u)-\e \nu(u)  \left(\frac{v}{u}\right)_x\right]_x \, dx \\
&=v_* \int_{-\ell}^\ell \left\{ -\frac{v  v_x}{u^2}+\frac{1}{u} \left(\frac{v^2}{u}\right)_x\right\} \, dx\\
&\quad +v_* \int_{-\ell}^\ell \frac{1}{u} \left[P(u)-\e \nu(u)  \left(\frac{v}{u}\right)_x\right]_x \, dx \\
&= v_* \int_{-\ell}^\ell  B_1(u,v)+ B_2(u,v)\, dx.
\end{aligned}
\end{equation}
We start by computing $B_1(u,v)$; we get
\begin{equation*}
\begin{aligned}
\int_{-\ell}^\ell B_1(u,v) \, dx &=   \int_{-\ell}^\ell\left\{\frac{1}{u}\left(\frac{v^2}{u}\right)_x - \frac{v v_x}{u^2}\right \}\, dx  \\
 &=  \int_{-\ell}^\ell \frac{1}{u}\left(\frac{2v v_x u-v^2 u_x}{u^2}\right) \, dx- \int_{-\ell}^\ell\frac{v^2 u_x}{u^3} \, dx  \\
  &=2  \int_{-\ell}^\ell \frac{vv_x}{u^2}dx-2 \int_{-\ell}^\ell\frac{v^2 u_x}{u^3} \, dx  \\
 &=  \int_{-\ell}^\ell \frac{(v^2)_x}{u^2}dx+ \int_{-\ell}^\ell v^2 \left(\frac{1}{u^2} \right)_x\, dx  \\
 &=   \left(\frac{2v^2}{2u^2}\right)\Big|^{\ell}_{-\ell} +\int_{-\ell}^\ell\frac{ u_x \, v^2}{u^3} \, dx+ \int_{-\ell}^\ell\frac{u_x \, v^2}{u^3} \, dx  \\
& =   \frac{2v_*^2}{u_+^2u_-^2} (u^2_--u^2_+)+\frac{2|v|^2_{{}_{L^\infty L^\infty}}}{u_-^3}(u_+-u_-),
\end{aligned}
\end{equation*}
where in the fifth equality we integrated by parts both terms; the first  term in the above sum is negative, while the second can be bounded via the difference $u_+-u_-$.

 We turn our attention to $B_2(u,v)$; by integrating by parts
\begin{equation*}
\begin{aligned}
\int_{-\ell}^\ell B_2&(u,v) \, dx = \\
&=\left\{ \frac{1}{u}\left[ P(u)-\e \nu(u)\left(\frac{v}{u}\right)_x\right]\right \}\Big|^{\ell}_{-\ell} + \int_{-\ell}^\ell \frac{ u_x}{u^2}\left[P(u)-\e \nu(u) \left(\frac{v}{u}\right)_x\right] \, dx\\
&= \frac{1}{u_+}\left[P(u_+)- \e\nu(u_+)\left(\frac{-  u_x(\ell,t) v_*}{u_+^2}\right)\right]-\frac{1}{u_-}\left[P(u_-)- \e\nu(u_-)\left(\frac{-  u_x(-\ell,t) v_*}{u_-^2}\right)\right] \\
&\quad +\int_{-\ell}^\ell \frac{ u_x}{u^2}\left[P(u)-\e \nu(u)  \left(\frac{v}{u}\right)_x\right] \, dx, \\
\end{aligned}
\end{equation*}
where we used again \eqref{neuman} to erase the terms $v_x(\pm \ell,t)$.
As concerning the terms at the boundary, we first notice that, since $u_-<u_+$
\begin{equation*}
\frac{P(u_+)}{u_+} - \frac{P(u_-)}{u_-} < 0,
\end{equation*}
while for the terms involving $u_x(\pm \ell,t)$
we proceed as in \eqref{uselater}.  
 Finally, for the last term in $B_2(u,v)$, we first observe that
\begin{equation*}
\int^{\ell}_{-\ell} \frac{ u_x}{u^2}P(u) \, dx \lesssim u_+- u_- ,
\end{equation*}
and we are thus left with 
\begin{equation*}
\begin{aligned}
-\int_{-\ell}^\ell \frac{ u_x}{u^2}\left[\e \nu(u)  \left(\frac{v}{u}\right)_x\right] \, dx &=\e \int_{-\ell}^\ell\left(\frac{u_x \nu(u)}{u^2}\right)_x \left(\frac{v}{u}\right) \, dx \\
&=\int_{-\ell}^\ell \left( \frac{(\nu(u)u_x)_x}{u^3}-\frac{u_x^2 \nu(u)}{u^4}\right) v \, dx \\
&=\int_{-\ell}^\ell \frac{u_{xx}\nu(u)}{u^3} v \, dx- \int_{-\ell}^\ell \frac{u u_x^2 \nu'(u)-u_x^2 \nu(u)}{u^4} v \, dx \\
&=B_{21}(u,v)+ B_{22}(u,v).
\end{aligned}
\end{equation*}
For $B_{22}(u,v)$, by taking advantage of the positive sign of the functions $ \nu(u), u$ and its first derivative, we can state
\begin{equation*}
B_{22}(u,v) \lesssim c_1 \int_{-\ell}^\ell u_x \, dx,
\end{equation*}
where the {positive constant $c_1$} depends, among others, on $u_\pm$, $|v|_{{}_{L^\infty L^\infty}}$ and $\displaystyle{| u|_{{}_{L^\infty H^1}}}$. We recall that these quantities are finite because of Theorem \ref{existence} (see, in particular, Lemma A.5 and Lemma A.9).
As concerning $B_{21}(u,v)$ we have, again by integration by parts
\begin{equation*}
\begin{aligned}
B_{21}(u,v)&=-\int_{-\ell}^\ell u_x \left(\frac{\nu(u)}{u^3} v\right)_x \, dx\\
&= -\int_{-\ell}^\ell u_x \left( \frac{ \nu'(u)v}{u^3}+\frac{\nu(u)v_x}{u^3}-\frac{3u_x \nu(u)v}{u^4}\right)\, dx \\
& \lesssim c_2 \int_{-\ell}^\ell u_x \, dx,
\end{aligned}
\end{equation*}
where, as before, the  {positive constant $c_2$} depends, among others, on $u_\pm$, $\displaystyle{ | v|_{{}_{L^\infty H^1}}}$  and $\displaystyle{ | u|_{{}_{L^\infty H^1}}}$. Hence, both
 $B_{21}(u,v)$ and $B_{22}(u,v)$  can be bounded via the difference $u_+-u_-$.

\subsubsection*{Computation fo the sign of $C(u,\bar u)$ and $D(u,\bar u)$} We finally compute the last two terms in \eqref{rimasti}; on one side we have
\begin{equation*}
\begin{aligned}
C(u,\bar u)=\frac{d}{dt}\int_{-\ell}^{\ell} & u \left( \int_{0}^{\bar u} \frac{P(z)}{z^2} \, dz \right)\, dx = -\int_{-\ell}^{\ell} & v_x \left( \int_{0}^{\bar u} \frac{P(z)}{z^2} \, dz \right)\, dx;
\end{aligned}
\end{equation*}
on the other side
\begin{equation*}
\begin{aligned}
D(u,\bar u)=\frac{d}{dt}\int_{-\ell}^\ell & P'(\bar u)\left(1-\frac{u}{\bar u}\right) \, dx =\int_{-\ell}^\ell  u_t \frac{P(\bar u)}{\bar u} \, dx
=- \int_{-\ell}^\ell  v_x \frac{P(\bar u)}{\bar u} \, dx.
\end{aligned}
\end{equation*}
In both cases, by integration by parts and by taking  advantage of the positivity of $P(s)$, $\bar u$ and their derivatives, we can bound these terms from above with $\int \bar u_x$, i.e via the difference $u_+-u_-$.

\subsubsection*{Conclusion} Summing up, we have shown that
\begin{equation}\label{conclusion}
\frac{d}{dt}   \int_{-\ell}^\ell \mathcal E_L(u,v, \bar u, \bar v) \, dx \leq C^- + C_{1}^+  (u_+ -u_-) + C_{2}^+  |u_x(\ell,t)-u_x(-\ell, t)|,
\end{equation}
where $C^-<0$ is a negative constant collecting all the negative terms appearing in the previous computations, while $C_{1}^+ $ and $C_{2}^+$ are positive constants.

By taking advantage of hypothesis {\bf H3}, we can thus choose $\delta_1$ and $\delta_2$ in such a way that the right hand side in \eqref{conclusion} is negative, and the proof is completed.

\end{proof}

As a consequence of Proposition \ref{dtLnegative}, we are finally able to prove Theorem \ref{teointro2}; we recall the result for completeness.
\begin{theorem}
Let assumptions {\bf H1-2-3} be satisfied. Then $(\bar u,\bar v)$, the unique steady state of system \eqref{NS}, is stable in the following sense:
for every  $T >0$ it holds
\begin{equation}\label{finalestabilita}
\sup_{0\leq t\leq T}|(u,v)(t)-(\bar u,\bar v)|_{{}_{L^2}} \leq |(\bar u,\bar v)- (u_0,v_0)|_{{}_{L^2}}. 
\end{equation}
Moreover
\begin{equation*}
|(u, v) -(\bar u, \bar v)| _{{}_{L^1 H^1}} \leq C_T,
\end{equation*}
with $0 < C_T \to +\infty$ if and only if $T\to+ \infty$.

\end{theorem}
\begin{proof}
We make use of Proposition \ref{dtLnegative}; recalling that in  \eqref{derivataL} we erased the positive term ${\bf B_\e}$, by integrating in time the relation $\frac{d}{dt} L(t) \leq 0$ we have
\begin{equation*}
L(t) + \e \int^t_0\int_{-\ell}^{\ell} \nu(u) \left[ \left( \frac{v}{u}\right)_x\right]^2 \, dx dt \leq L(0).
\end{equation*}
On one side, by using the very definition of $L(t)$, the inquality $L(t) \leq L(0)$ implies
\begin{equation*}
|u-\bar u|_{{}_{L^2}} + |v-\bar v|_{{}_{L^2}} \leq C \left(|\bar u- u_0|_{{}_{L^2}} + |\bar v- v_0|_{{}_{L^2}}\right).
\end{equation*}
On the other side, we have the inequality
\begin{equation}\label{lyafinale1}
\begin{aligned}
C \int^t_0\int_{-\ell}^{\ell} \left[  \left( \frac{v}{u}\right)_x\right]^2 \, dx dt\leq \e \int^t_0\int_{-\ell}^{\ell}  \nu(u) \left[  \left( \frac{v}{u}\right)_x\right]^2 \, dx dt \leq L(0),
\end{aligned}
\end{equation}
and
\begin{equation*}
\begin{aligned}
\int_{-\ell}^{\ell} \left[ \left( \frac{v}{u}\right)_x\right]^2 \, dx  &=\int_{-\ell}^{\ell} \left\{ \frac{v_x^2}{u^2}+ \frac{v^2 u_x^2}{u^4}-\frac{(v^2)_x  u_x}{u^3} .\right\} \, dx\\
 \end{aligned}
\end{equation*}
Hence, \eqref{lyafinale1} becomes
\begin{equation*}
\begin{aligned}
\int_0^t\int_{-\ell}^{\ell} \frac{ v_x^2}{u^2} \, dx dt+ \int_0^t\int_{-\ell}^{\ell}\frac{v^2 u_x^2}{u^4} \, dx dt& \leq L(0)+ 2\int_0^t \int_{-\ell}^{\ell} \frac{v v_x u_x}{u^3}\, dxdt \\
& \leq L(0) + C_T.
\end{aligned}
\end{equation*}
In particular, recalling that $ \bar v_x=0$ and since the second integral on the left hand is positive, we can thus state that
\begin{equation*}
\begin{aligned}
| v_x - \bar v_x| _{{}_{L^1L^2}}=\int_0^t\int_{-\ell}^{\ell}  v^2_x \, dx dt& \leq \left(|\bar u- u_0|_{L^2} + |\bar v- v_0|_{L^2}\right) + C_T.\\
\end{aligned}
\end{equation*}
Moreover
\begin{equation*}
\begin{aligned}
\int_0^t \int_{-\ell}^{\ell}(u_x- \bar u_x)^2 \, dx dt &\leq 2\int_0^t \int_{-\ell}^{\ell}  u_x^2+ \int_0^t \int_{-\ell}^{\ell} \bar u_x^2 \, dx dt \\
& \leq L(0) + C_T
\end{aligned}
\end{equation*}
implying
\begin{equation*}
| u_x - \bar u_x| _{{}_{L^1L^2}}  \leq  |\bar u- u_0|_{L^2} + |\bar v- v_0|_{L^2} + C_T.
\end{equation*}
The proof is now complete.
\end{proof}
We point out that estimate \eqref{finalestabilita} implies  stability of the steady state in the sense that that initial data $(u_0,v_0)$ close to the steady state in the $L^2$-norm will generate a solution $(u,v)$ to \eqref{NS} which is still close to $(\bar u, \bar v)$ in the $L^2$-norm, for all $t \geq 0$.

\appendix
\section{Existence and regularity of the solution}

We here discuss  existence and regularity of the solutions to the Navier-Stokes equations \eqref{PLgenerico2}; we write the problem  in terms of
the variables mass density and velocity of the fluid $(\rho, w)$
\begin{equation}\label{A10}
 \left\{\begin{aligned}
& \rho_t +(\rho\, w)_x=0 \\
& (\rho \, w)_t+ \left( \rho \, w^2+ P(\rho)\right)_x= \varepsilon \left( \nu(\rho)  w_x \right)_x, \\
   \end{aligned}\right. 
\end{equation}
and, for simplicity, we consider
$P(\rho)= \k \rho^{\g}$, $\gamma>1$ and $\nu(\rho)=1$. We recall that system \eqref{A10} is considered in the bounded interval $I=(-\ell,\ell)$ with boundary conditions
\begin{equation}\label{BCappendix}
\rho(-\ell)=\rho_->0, \quad w(\pm \ell)= w_\pm >0,
\end{equation}
and it is subject to the initial datum $(\rho,w)(x,0)=(\rho_0,w_0)$.
\vskip0.2cm
\noindent 
{\bf Notations:} As before, we will denote with 
$|f|_{{}_{XY}} := |f|_{{}_{X([0,T], Y(I))}}.$
If not specified otherwise, we will denote with
\begin{equation*}
\int f := \int_{-\ell}^\ell f(x,t) \, dx.
\end{equation*}

\vskip0.2cm

\begin{theorem}\label{existence}
Let us consider the Cauchy problem for \eqref{A10}, and let assume $\rho_0(x) := \rho(x, 0) \in H^1(I)$ and $w_0(x) := w(x, 0) \in H^1(I)$. Then there exists a unique solution $(\rho, w)$ to \eqref{A10}  satisfying the boundary conditions \eqref{BCappendix} and such that
\begin{equation*}
\rho \in L^\infty([0,T], H^1(I)) \quad {\rm and} \quad  w \in L^\infty([0,T], H^1(I)) \cap L^2([0,T], H^2(I)).
\end{equation*}
\end{theorem}

The proof of Theorem \ref{existence} follows from the proof of several Lemmas.

\begin{lemma}\label{A1}

There exists $T >0$ such that, for all $t \in [0,T]$, there holds
\begin{equation*}
|\sqrt{\rho} w|^2_{{}_{L^\infty L^2}}+ |{\rho} |^{\g}_{{}_{L^\infty L^\g}} + |w_x|_{{}_{L^2L^2}} \leq C_T,
\end{equation*}
for some constant $C_T>0.$
\end{lemma}

\begin{proof}
By combining the two equations in \eqref{A10} we get
\begin{equation*}
\rho w_t + \rho w w_x + P_x = \varepsilon w_{xx}.
\end{equation*}
By integrating in space over $I$ and by multiplying by $w$
\begin{equation*}
\frac{1}{2}\int \rho (w^2)_t-\frac{1}{2}\int w^2 (\rho w)_x + \frac{1}{2}\rho w^3\big|_I + \int P_x w- \int \varepsilon w_{xx} w=0,
\end{equation*}
where we also integrated by parts. We now observe that
\begin{equation*}
\begin{aligned}
\int P_x w &= \g \k \rho^{\g-1} \left( -\rho w_x-\rho_t\right)\\
&=-\frac{d}{dt}\int \k \rho^{\g}-\int\k \g\rho^{\g} w_x \\
&=-\frac{d}{dt}\k\int  \rho^{\g}+\int \g P_x w -\g P w\big|_{I},
\end{aligned}
\end{equation*}
implying
\begin{equation*}
\int P_x w =\frac{\k}{\g-1}\frac{d}{dt}\int  \rho^{\g}-\frac{\g}{\g-1}P w\big|_{I}.
\end{equation*}
Summing up, recalling the energy formula $E(\rho(x,t),w(x,t))= \frac{1}{2}\rho w^2+ \frac{\k}{\g-1}\rho^\g$, we have
\begin{equation*}
\frac{d}{dt} \int E + \e \int  w_x^2 = C.
\end{equation*}
Finally, integrating in time we get
\begin{equation*}
\int \left(E(t)-E(\rho(0),w(0))\right) \, dt + \e \int_0^t \int w_x^2 \, dt=C t,
\end{equation*}
implying
\begin{equation}\label{miserve}
\sup_{t\in[0,T] }|\sqrt{\rho} w|_{{}_{L^2}}^2 +\sup_{t \in[0,T]} |\rho|^\g_{{}_{L^\gamma}} + \e \int_0^T \int w_x^2 \, dt \leq C_T.
\end{equation}
\end{proof}
\begin{lemma}\label{A2}
There holds
\begin{equation}\label{stimalinftyrho}
|\rho|_{{}_{L^\infty L^\infty}} \leq C.
\end{equation}
\end{lemma}

\begin{proof}
Consider the Lagrangian flow ${X}=X(x,t)$ of $w$, defined as
\begin{equation}\label{ALF}
 \left\{\begin{aligned}
& \frac{\partial X}{\partial t}=w(X(x,t),x),\\
& X(x,0)=x \in [-\ell,\ell].\\ 
   \end{aligned}\right. 
\end{equation}
In order to prove \eqref{stimalinftyrho} we thus need to prove that
\begin{equation*}
\rho(X(x,t),x) \leq C,
\end{equation*}
for any $(x,t) \in (0,T] \times (-\ell,\ell)$ and for some constant $C \geq 0$. Fixed $t_0 \in (0,T]$ and given the initial mass $\int \rho_0  := m_0 \leq C$, because of the conservation of the mass and from the very definition of the Lagrangian flow  we can find $x_1 \in (-\ell,\ell)$ such that
$$\rho_0(x_1) \geq C^{-1} \quad \mbox{and} \quad \rho(X(x_1,t_0),t_0) \leq C.$$
We now want to prove that, for any $x_2 \in (-\ell,\ell)$, $\rho(x_2, X(x_2,t_0)) \leq C$; we let $X_j:= X(x_j,t)$ for $j=1,2$ and we define
$$F(t)=\log (\rho(X_2,t))-\log (\rho(X_1,t)).$$
By using \eqref{ALF} we have
\begin{equation}\label{1LF}
\begin{aligned}
\frac{dF}{dt}&=\frac{1}{ \rho(X_2,t)}\left(  \rho_t(X_2,t)+\rho_x(X_2,t) \frac{dX_2}{dt}\right)-\frac{1}{ \rho(X_1,t)}\left(  \rho_t(X_1,t))+\rho_x(X_1,t) \frac{dX_1} {dt}\right) \\
& =\frac{1}{ \rho(X_2,t)}\left(-\rho_x(X_2,t)w(X_2,t)-\rho(X_2,t)w_x(X_2,t)+\rho_x(X_2,t) w(X_2,t)\right) \\
& \quad -\frac{1}{ \rho(X_1,t))}\left(-\rho_x(X_1,t)w(X_1,t)-\rho(X_1,t)w_x(X_1,t)+\rho_x(X_1,t) w(X_1,t)\right) \\
&=-w_x(X_2,t)+w_x(X_1,t) \\
&=-\int_{X_1}^{X_2} w_{xx} \, dx\\
&=-\frac{1}{\e}\int_{X_1}^{X_2} (\rho w_t+\rho w w_x + P_x) \, dx.
\end{aligned}
\end{equation}
Let now $$V(t)=\int_{X_1}^{X_2} \rho w \, dx,$$
so that
\begin{equation}\label{2LF}
\begin{aligned}
\frac{dV}{dt}&= (\rho w)(X_2,t)\frac{dX_2}{dt}-(\rho w)(X_1,t)\frac{dX_1}{dt}+ \int_{X_1}^{X_2} (\rho w)_t \, dx\\
&=(\rho w^2)(X_2,t)-(\rho w^2)(X_1,t)+ \int_{X_1}^{X_2} [ -(\rho w)_x w+\rho w_t] \, dx \\
&= \int_{X_1}^{X_2}[ (\rho w^2)_x -(\rho w)_x w+\rho w_t] \, dx \\
&=  \int_{X_1}^{X_2}[ \rho w w_x + \rho w_t] \, dx.
\end{aligned}
\end{equation}
Substituting \eqref{2LF} into \eqref{1LF}  we get
\begin{equation}\label{3LF}
\begin{aligned}
\e\frac{dF}{dt}+\frac{dV}{dt}= - 	 \int_{X_1}^{X_2} P_x \, dx.
\end{aligned}
\end{equation}
By setting 
$$\alpha(t)=\frac{P(\rho(X_1,t))-P(\rho(X_2,t))}{\e F(t)} \geq 0,$$
equation \eqref{3LF} can be rewritten as
\begin{equation*}
\e\frac{dF}{dt}+\frac{dV}{dt}= -\alpha(\e F+V)+\alpha V,
\end{equation*}
with solution given by
\begin{equation*}
\e F(t)+ V(t)= e^{-\int_0^t \alpha(s) \, ds} (\e F(0)+V(0))+ \int_0^t e^{-\int_0^s \alpha(\tau) \, d\tau} \alpha(s) V(s) \, ds.
\end{equation*}
Since $F(0) \leq |\rho_0|_{{}_{L^\infty}}$, we have, for any $t \geq 0$
\begin{equation}\label{4LF}
\begin{aligned}
\e F(t)&\leq \e C+V(0) + |V(t)|+ \int_0^t e^{-\int_0^s \alpha(\tau) \, d\tau} \alpha(s) |V(s)| \, ds \\
&\leq \e C+V(0) + |V(t)|+ \sup_{0\leq t \leq t_0} |V(t)|\int_0^t e^{-\int_0^s \alpha(\tau) \, d\tau} \alpha(s) \, ds \\
& \leq \e C+V(0) + |V(t)|+ \sup_{0\leq t \leq t_0} |V(t)|\,e^{-\int_s^t \alpha(\tau)}\, d\tau \Big|^{t}_{0} \\
& \leq \e C+V(0) + |V(t)|+ \sup_{0\leq t \leq t_0} |V(t)|,
\end{aligned}
\end{equation}
where we used the fact that $e^{-\int_0^t \alpha(s) \, ds}  \leq 1$ for all $t \geq 0$.
We now observe that
\begin{equation*}
\begin{aligned}
|V(0)| &\leq \int_{X_1}^{X_2} \rho_0 |w_0| \, dx\leq C, \\
|V(t)| & \leq  \int_{X_1}^{X_2} \rho |w| \, dx \leq \left(  \int_{X_1}^{X_2} \rho \, dx \right)^{1/2} \left( \int_{X_1}^{X_2}\rho w^2 \, dx \right)^{1/2},
\end{aligned}
\end{equation*}
implying
\begin{equation*}
\begin{aligned}
 \sup_{0\leq t \leq t_0} |V(t)| & \leq \sup_{0\leq t \leq t_0}\left(  \int \rho \right)^{1/2} \left( \int\rho w^2 \right)^{1/2} \leq C,\\
\end{aligned}
\end{equation*}
where the quantity on the right hand side is bounded by a constant because of   \eqref{miserve}. Combining the above estimates with \eqref{4LF}, we finally obtain $\e F(t_0) \leq C$, implying
\begin{equation*}
\log (\rho(X(t_0,x_2),t_0))=\log (\rho(X(t_0,x_1),t_0))+ F(t_0) \leq C.
\end{equation*}
The claim then follows because of the arbitrariness of $x_2$ and $t_0$.

\end{proof}

\begin{remark}\label{normapiccola}{\rm
We point out that, because of assumption {\bf H3} on the boundary data, it necessary has to be $|\rho_0|_{_{L^	\infty}} < \delta_1$. In particular this implies that  the constant $C$ in \eqref{stimalinftyrho} is less than $\delta_1 >0$.
}
\end{remark}

\begin{lemma}\label{A3} There exists $T >0$ such that, for all $t \in [0,T]$, there holds
\begin{equation*}
\int_0^T |\sqrt{\rho} w_t|^2_{{}_{L^2}} \, dt + |w_x|^2_{{}_{L^2}} \leq C_T \quad \mbox{and} \quad \sup_{0 \leq t \leq T} \left( |w|_{{}_{L^\infty}}+ |w_x|_{{}_{L^2}}\right) \leq C_T,
\end{equation*}
for some constant $C_T>0.$
\end{lemma}
\begin{proof}
We multiply the second  equation in \eqref{A10} by $w_t$ \and we integrate in space. We get
\begin{equation*}
\begin{aligned}
\int \rho_t w w_t + \rho w_t^2 + (\rho w^2)_x w_t + P_x w_t &= \int \rho w_t^2 + \int \rho w w_x w_t + \int P_x w_t \\
&= \int \e w_{xx} w_t.
\end{aligned}
\end{equation*}
We also have
\begin{equation*}
\begin{aligned}
(\rho w_t) w w_x &= (-\rho_t w + \e w_{xx}-(\rho w^2)_x -P_x) w w_x \\
& = ((\rho w)_x w + \e w_{xx}-(\rho w^2)_x -P_x)w w_x \\
& =\rho w^2 w_x^2 + { \e w_{xx} w w_x - P_x w w_x}, \\
&=\rho w^2 w_x^2+{ G_x w w_x+ (\e-1)w_{xx} w w_x},
\end{aligned}
\end{equation*}
where $G:= w_x -P$.
The previous equality implies
\begin{equation}\label{imp}
\begin{aligned}
\int \rho w_t^2 + \e \frac{d}{dt}\int \frac{1}{2} w^2_x&\leq \e  w_x w_t\big|_{I} + \int P w_{xt} - P w_t\big|_I \\
& \quad+ \int \rho w^2 w_x^2 +{ \int G_x w w_x+ (\e-1)\int w_{xx} w w_x} \\
& \leq C+ \int \rho w^2 w_x^2 + \int P w_{xt} +  { \int G_x w w_x }. \\
\end{aligned}
\end{equation}
Going further
\begin{equation*}
\begin{aligned}
\int G_x w w_x \leq \left( \int G_x^2\right)^{1/2} \left( \int w^2 w_x^2\right)^{1/2}& \leq |G_x|_{{}_{L^2}} |w|_{{}_{L^\infty}}\left( \int w_x^2\right)^{1/2} \\
& \leq |G_x|_{{}_{L^2}} |w_x|_{{}_{L^2}} |w_x|_{{}_{L^2}}\\
& \leq |G_x|_{{}_{L^2}} |w_x|^2_{{}_{L^2}} .
\end{aligned}
\end{equation*}
Moreover, a straightforward computation shows that
\begin{equation*}
\begin{aligned}
&\int P w_{xt}= \frac{d}{dt} \int P w_x - \int Pw (G_x + P_x) + (\gamma-1)\int (G+P)^2 P, \\
&\frac{1}{2(2\gamma-1)}\frac{d}{dt}\int P^2 = \int P w P_x + C,
\end{aligned}
\end{equation*}
and
\begin{equation*}
(\g-1)\int P w_x^2 = (\g-1)\int P G^2-4(\g-1)\int P P_x w -(\g-1)\int P^3 + C,
\end{equation*}
where we used the explicit expression fo the pressure  $P(\rho)= \k \rho^\g$.
By using the previous identities  we finally obtain
\begin{equation*}
\int P w_{xt} =C+ \frac{d}{dt}\int P w_x- \int P wG_x + (\g-1) \int  P(G^2-P^2)  
 -\frac{4\g-3}{2(2\g-1)}\frac{d}{dt}\int P^2.
\end{equation*}
We can thus integrate in time \eqref{imp}, obtaining
\begin{equation}\label{ineqcompleta}
\begin{aligned}
\int_0^t\int \rho w_t^2 \, dt + \e \frac{d}{dt} \int_0^t\int  w_x^2 \, dt &\leq C +  \int_0^t\int \rho w^2 w_x^2 \, dt  \\
& \quad +{ \int_0^t  |G_x|_{{}_{L^2}} |w_x|^2_{{}_{L^2}}  \, dt} \\
& \quad + (\g-1) \int_0^t \int PG_2 \, dt - (\g-1)\int_0^t \int P^3 \, dt \\
& \quad + \int_0^t \int P |w| |G_x| \, dt \\
& \quad + \int \left( P w_x(t)-Pw_x(0)\right) -\frac{4\g-3}{2(2\g-1)} \int [P^2(t)-P^2(0)]. \\
\end{aligned}
\end{equation}
There hold the following estimates for the terms appearing on the right hand side of 	\eqref{ineqcompleta}
\begin{equation*}
\begin{aligned}
{ \int_0^t  |G_x|_{{}_{L^2}} |w_x|^2_{{}_{L^2}}  \, dt } &\leq  \left( \int_0^t |G_x|_{{}_{L^2}}^2 \, dt\right)^{1/2}  \left( \int_0^t |w_x|_{{}_{L^2}}^4 \, dt\right)^{1/2}\\
& \lesssim  \left( \int_0^t |G_x|_{{}_{L^2}}^2 \, dt\right)+ \left( \int_0^t |w_x|_{{}_{L^2}}^4 \, dt\right); \\
- (\g-1)\int_0^t \int P^3 \,dt &<0; \\
  \int [P^2(t)-P^2(0)] =  \int (\rho^{2\gamma}-\rho_0^{2\gamma}) &\leq C \qquad \mbox{for Lemma \ref{A2}}; \\
  \int Pw_x(0) &\leq C \qquad \mbox{for the regularity of the initial data;}  \\
  \int P w_x(t) & \leq |\rho|_{{}_{L^\infty}}^{\gamma} \int w_x^2. \qquad
\end{aligned}
\end{equation*}
Inequality \eqref{ineqcompleta} can be thus rewritten as
\begin{equation}\label{ultimastima}
\begin{aligned}
\int_0^t |\sqrt{\rho} w_t|^2_{{}_{L^2}} \, dt + \e | w_x|^2_{{}_{L^2}} &\leq C+\int_0^t\int (\rho w^2 w_x^2 + PG^2 + P |w| |G_x|) \, dt\\
& \leq  C+ \left(|\rho|_{{}_{L^\infty L^\infty}} {+1}\right) \int_0^t |w_x|^4_{{}_{L^2}} \, dt \\
& \quad +|\rho|^{\g+1/2}_{{}_{L^\infty L^\infty}}\sup_t |\sqrt{\rho}w|_{{}_{L^2}} \int_{0}^t |w_x|^{2}_{{}_{L^2}} \, dt \\
&\quad +  |\rho|_{{}_{L^\infty L^\infty}} \int_0^t |\sqrt{\rho} w_t|^2_{{}_{L^2}} \, dt		\\
&\quad + |\rho|^{\g}_{{}_{L^\infty L^\infty}}\sup_t |\sqrt{\rho}w|_{{}_{L^2}} \int_{0}^t |\sqrt{\rho} w_t|_{{}_{L^2}} \, dt \\
& \quad +  |\rho|_{{}_{L^\infty}}^{\gamma} |w_x|^2_{{}_{L^2}} ,
\end{aligned}
\end{equation}
where we used
\begin{equation*}
\begin{aligned}
\int_0^t\int \rho w^2 w_x^2  \, dt&\leq \int_0^t |\rho|_{{}_{L^\infty}} |w|^2_{{}_{L^\infty}} |w|^2_{{}_{L^2}} \leq |\rho|_{{}_{L^\infty L^\infty}} \int_0^t |w_x|^4_{{}_{L^2}} \, dt; \\
\int_0^t\int PG^2 \, dt &\leq \int_0^t\int P(w_x-P)^2 \, dt \leq \int_0^t\int P (w_x^2+P^2) \, dt\leq \int_0^t\int P^3 \, dt + \int_0^t\int  P^2 w_x^2 \, dt\\
&  \leq |P|^{3}_{{}_{L^\infty L^\infty}}+  |P|^{2}_{{}_{L^\infty L^\infty}}\int_0^t\int w_x^2 \, dt \leq C \qquad \mbox{ for Lemmas \ref{A1} and \ref{A2}}; \\
 |G_x|_{{}_{L^2}} &= |\sqrt{\rho}(\sqrt{\rho} w_t + \sqrt{\rho} w w_x)|_{{}_{L^2}} \\
 &\leq |\rho|_{{}_{L^\infty}}^{1/2} \left( |\sqrt{\rho} w_t|_{{}_{L^2}} + |\sqrt{\rho} w w_x|_{{}_{L^2}}\right) \\
 & \leq |\rho|_{{}_{L^\infty}}^{1/2} \left( |\sqrt{\rho} w_t|_{{}_{L^2}} +|{\rho}|_{{}_{L^\infty}}^{1/2} |  w_x|^{2}_{{}_{L^2}}\right);  \\
 \int_0^t |G_x|^2_{{}_{L^2}} & \leq  |\rho|_{{}_{L^\infty L^\infty}} \int_0^t |\sqrt{\rho} w_t|^2_{{}_{L^2}} \, dt+ |\rho|_{{}_{L^\infty L^\infty}} \int_0^t 	 |w_x|^4_{{}_{L^2}} \, dt;	\\
 \int_0^t\int P |w||G_x| \, dt &\leq \int_0^t |\rho|_{{}_{L^\infty}}^{\gamma-1/2} |\sqrt{\rho} w|_{{}_{L^2}} |G_x|_{{}_{L^2}}  \\
 & \leq |\rho|_{{}_{L^\infty L^\infty}}^\gamma \sup_{t} |\sqrt{\rho} w|_{{}_{L^2}} \int_0^t |\sqrt{\rho} w_t|_{{}_{L^2}} \, dt +   |\rho|_{{}_{L^\infty L^\infty}}^{\gamma+1/2} \sup_{t} |\sqrt{\rho} w|_{{}_{L^2}} \int_0^t |w_x|^{2}_{{}_{L^2}} \, dt.
\end{aligned}
\end{equation*}
Recalling that $|\rho|_{{}_{L^\infty}} < \delta$ by Remark \ref{normapiccola}, \eqref{ultimastima} becomes
\begin{equation}\label{ultimastima2}
\begin{aligned}
(1-\e)\int_0^t |\sqrt{\rho} w_t|^2_{{}_{L^2}} \, dt + (\e-\delta^\gamma) | w_x|^2_{{}_{L^2}}  \, dt
& \leq  C+ {\left(|\rho|_{{}_{L^\infty L^\infty}}{+1}\right)} \int_0^t |w_x|^4_{{}_{L^2}} \,  \\
& \quad + |\rho|^{\g+1/2}_{{}_{L^\infty L^\infty}}\sup_t |\sqrt{\rho}w|_{{}_{L^2}} \int_{0}^t |w_x|^{{2}}_{{}_{L^2}} \, dt \\
&\quad + |\rho|^{\g}_{{}_{L^\infty L^\infty}}\sup_t |\sqrt{\rho}w|_{{}_{L^2}} \int_{0}^t |\sqrt{\rho} w_t|_{{}_{L^2}} \, dt. \\
\end{aligned}
\end{equation}
By applying the Gronwall's inequality to \eqref{ultimastima2} we thus end up with
\begin{equation}\label{primastima}
\int_0^t |\sqrt{\rho} w_t|^2_{{}_{L^2}} \, dt+  | w_x|^2_{{}_{L^2}} \leq C_T,
\end{equation}
providing $\delta<\e^{1/\gamma}$.
Moreover, by passing to the sup in time for $t \in [0,T]$ in \eqref{primastima} we also have
\begin{equation}\label{regolaritau}
\sup_{0 \leq t \leq T} | w_x|^2_{{}_{L^2}} \leq C_T,
\end{equation}
and, recalling that $|w|_{{}_{L^\infty}} \leq C |w_x|_{{}_{L^2}}$ by Sobolev embedding, the proof is complete.

\end{proof}

\begin{lemma}\label{A4}  There exists $T >0$ such that, for all $t \in [0,T]$ and for some constant $C_T>0$, there holds
\begin{equation*}
\int_0^T \left( |w_x|^2_{{}_{L^\infty}} + |G_x|^2_{{}_{L^2}}\right) \, dt \leq C_T,
\end{equation*}
 where $G_x:=w_{xx}-P_x$.
\end{lemma}

\begin{proof}
From the very definition of $G$ we have
\begin{equation*}
|w_x|^2_{{}_{L^\infty}}  \leq 2|G|_{{}_{L^\infty}}^2 +2 |P|^2_{{}_{L^\infty}} \leq 2|G_x|_{{}_{L^2}}^2 + 2|P|^2_{{}_{L^\infty}}.
\end{equation*}
Moreover, on one side
\begin{equation*}
\int_0^T|G|^2_{{}_{L^2}} \, dt  \leq  \int_0^T\int |w_x|^2  dt +\int_0^T\int  |P|^2 dt \leq C_T,
\end{equation*}
because of Lemma \ref{A1}. On the other side
\begin{equation*}
\begin{aligned}
G_x = w_{xx}-P_x = (\rho w)_t + (\rho w)_x &= \rho w_t +\rho_t w + \rho_x w^2 + 2 \rho w w_x \\
&= \rho w_t-\rho_x w^2 + \rho  w w_x  +\rho_x w^2 + 2 \rho w w_x,
\end{aligned}
\end{equation*}
implying
\begin{equation*}
\int_0^T|G_x|^2_{{}_{L^2}} \, dt  \leq \int_0^T \left( |\sqrt{\rho} w_t|^2_{{}_{L^2}}  + |w w_x|^2_{{}_{L^2}} \right) \, dt\leq C +  \int_0^T |w|^2_{{}_{L^\infty}}|w_x|^2_{{}_{L^2}}\, dt \leq C_T,
\end{equation*}
again because of Lemma \ref{A1} and Lemma \ref{A3}. The claim then follows.

\end{proof}
\begin{remark}\label{perlyapunov}
We observe that, since $w_{xx}=G_x+P_x$,  Lemma \ref{A4} also implies $|w_{xx}|_{{}_{L^2L^2}} \leq C_T$. This, together with Lemma \ref{A1}, implies that $|w|_{{}_{L^2H^2}} \leq C_T$.

\end{remark}

\begin{lemma}\label{A5}   There exists $T >0$ such that, for all $t \in [0,T]$  there holds
\begin{equation*}
\sup_{0 \leq t \leq T} |\rho|_{{}_{L^2}} \leq C_T,
\end{equation*}
for some constant $C_T>0$.
\end{lemma}

\begin{proof}
Multiplying by $\rho$ the first equation in \eqref{A10} and integrating in space we get
\begin{equation*}
\begin{aligned}
\int \rho_t \rho = -\int \rho_x w \rho -\int\rho w_x \rho &= \int \rho (w \rho)_x-\int \rho^2 w_x + C \\
&=\int \rho w \rho_x + C.
\end{aligned}
\end{equation*}
Hence
\begin{equation*}
\frac{1}{2}\frac{d}{dt} \int \rho^2 \leq \frac{1}{2 } \int \rho^2 |w_x| + C,
\end{equation*}
implying, after integrating in time
\begin{equation*}
\begin{aligned}
\int \rho^2 &\leq \int_0^t \int |\rho^2| |w_x| \, dt+ C t\\
& \leq \int_0^t \left[\left( \int \rho^4\right)^{1/2} \left( \int w_x^2\right)^{1/2}\right] + Ct \\
& \leq |\rho|_{{}_{L^\infty L^\infty}}^4 \int_0^t |w_x|_{{}_{L^2}} + C t,
\end{aligned}
\end{equation*}
and the claim follows from Lemma \ref{A3}.

\end{proof}

\begin{lemma}\label{A6} There exists $T >0$ such that, for all $t \in [0,T]$  there holds
\begin{equation*}
\sup_{0 \leq t \leq T} |\rho_x|_{{}_{L^2}} \leq C_T,
\end{equation*}
for some constant $C_T>0$.
\end{lemma}

\begin{proof}
Let us differentiate with respect to $x$ the first equation in \eqref{A10}; by multiplying it by $\rho_x$ and integrating over $I$ we get
\begin{equation}\label{1A5}
\frac{1}{2}\frac{d}{dt} \int \rho_x^2 + \int (\rho_x w + \rho w_x  )_x \rho_x =0.
\end{equation}
Moreover
\begin{equation}\label{2A5}
 \int (\rho_x w )_x \rho_x  = \int \rho_{xx} \rho_x w+ \int \rho_x^2 w_x
\end{equation}
and 
\begin{equation*}
\int \rho_{xx} \rho_x w= - \int \rho_x (\rho_x w)_x + C = C - \int \rho_x \rho_{xx} w -\int{\rho}_x^2 w_x,
\end{equation*}
implying
\begin{equation*}
\int \rho_{xx} \rho_x w = \frac{C}{2}- \frac{1}{2}\int \rho_x^2 w_x.
\end{equation*}
Hence, \eqref{2A5} becomes
\begin{equation*}
\int (\rho_x w )_x \rho_x= \frac{C}{2}+\frac{1}{2}\int \rho_x^2 w_x.
\end{equation*}
Going further
\begin{equation*}
 \int (\rho w_x  )_x \rho_x  = \int \rho_x^2 w_x + \int \rho w_{xx} \rho_x,
 \end{equation*}
and
\begin{equation*}
\begin{aligned}
 \int \rho w_{xx} \rho_x = \int \rho (G_x+P_x) \rho_x& \leq \int\rho^2 (G_x+P_x)^2 + \int \rho_x^2 \\
 & \leq |\rho|^2_{{}_{L^\infty}} \left(\int G_x^2 +  \int P_x^2 \right)+ \int \rho_x^2 \\
 & \leq  |\rho|^2_{{}_{L^\infty}} \int G_x^2+  |\rho|^{2\gamma}_{{}_{L^\infty}} \int \rho_x^2+\int \rho_x^2.
 \end{aligned}
\end{equation*}
Collecting all the above estimates, \eqref{1A5} thus becomes
\begin{equation*}
\frac{1}{2}\frac{d}{dt} \int \rho_x^2 \leq \left(|w_x|_{{}_{L^\infty}} +  |\rho|^{2\gamma}_{{}_{L^\infty}}+ 1\right) \int \rho_x^2 + \int G_x^2.
\end{equation*}
and the claim follows  after integrating in time from an application of Gronwall's inequality, recalling that  $\int_0^t |w_x|^2_{{}_{L^\infty}}$ and $\int_0^t |G_x|^2_{{}_{L^2}} $ are bounded because of Lemma \ref{A4}.

\end{proof}

Combining Lemmas \ref{A5} and \ref{A6} we get $\rho \in L^\infty([0,T], H^1(I))$; this, together with estimate \eqref{regolaritau}, completes the proof of Theorem \ref{existence}.


\baselineskip=16pt
\begin{thebibliography}{99}

\bibitem{BasCorAN}
     \newblock G. Bastin, J.M. Coron and B. D'Andr\'ea-Novel,
     \newblock \emph{ On Lyapunov stability of linearized Saint-Venant equations for a sloping channel},
     \newblock Netw. Heterog. Media \textbf{4} (2009), no. 2, 177--187.
  
 


     
 \bibitem{CBA}
     \newblock G. Bastin, J.M. Coron and B. D'Andr\'ea-Novel,
     \newblock \emph{ Dissipative boundary
conditions for one dimensional nonlinear hyperbolic systems},
\newblock SIAM Journal on Control and Optimization, {\bf 47} (2008) no. 3,1460--1498.

\bibitem{CBA2}
     \newblock G. Bastin, J.M. Coron and B. D'Andr\'ea-Novel,
     \newblock \emph{A strict Lyapunov function for
boundary control of hyperbolic systems of conservation laws},
\newblock IEEE Transactions on Automatic Control {\bf 52} (2007), no. 1, 2--11.


 \bibitem{BW}   
    \newblock M. Beck, C.E. Wayne,
    \newblock{ Metastability and rapid convergence to quasi-stationary bar states
for the 2D Navier-Stokes equations},
\newblock Proc. Roy. Soc. Edinburgh Sect. A. \textbf{143} (2013), 905--927.

\bibitem{BreDes07}
\newblock D. Bresch,  B. Desjardins,
\newblock\emph{ On the existence of global
weak solutions to the Navier-Stokes equations for viscous compressible and heat conducting fluids},
\newblock J. Math. Pures Appl. {\bf 87} (2007), 57--90.
     
\bibitem{BDG}     
   \newblock   D. Bresch, B, Desjardins, D. G\'erard-Varet,
   \newblock \emph{ On compressible Navier-Stokes equations with density dependent viscosities
in bounded domains},
   \newblock J. Math. Pures Appl. \textbf{(9) 87} (2007), no. 2, 227--235.

\bibitem{BrDeMe06}
     \newblock D. Bresch, B. Desjardins B. and G. M\'etivier, 
     \newblock \emph{ Recent Mathematical Results and Open Problems about Shallow Water Equations},
     \newblock Analysis and Simulation of Fluid Dynamics, Series in Advances in Mathematical Fluid Mechanics, Birkhauser Basel, (2006), pp. 15--31.
     

\bibitem{CZ}
\newblock Z. Chen, H. Zhao,
 \newblock \emph{Asymptotics of the 1D compressible Navier-Stokes equations with density-dependent viscosity},
 \newblock  J. Differential Equations \textbf{269} (2020), no. 1, 912--953. 


     
  \bibitem{CK}   
 \newblock H.J. Choe, H. Kim,
 \newblock \emph{Strong solutions of the Navier-Stokes equations for isentropic compressible fluids},
 \newblock  J. Differential Equations \textbf{190} (2003), no. 2, 504--523. 

\bibitem{Daf97}
C.M. Dafermos,
     \newblock \emph{ ``Hyperbolic Systems of Conservation Laws"},
\newblock Springer Verlag, New York, 1997.

\bibitem{Des}
\newblock B. Desjardins,
\newblock \emph{Regularity of weak solutions of the compressible isentropic Navier-Stokes equations},
\newblock{Comm. PDE, \textbf{22} (1997),  977--1008-}

\bibitem{DiBaCo}
 \newblock A. Diagne, G. Bastin and J.M. Coron,
 \newblock \emph { Lyapunov exponential stability of linear hyperbolic systems of balance laws},
 \newblock Preprint of the 18th IFAC World Congress, Milano (Italy) August 28-September 2, 2011.
 
 \bibitem{Evans}
  \newblock L.C. Evans,
   \newblock \emph {Partial Differential Equations}
   \newblock American Mathematical Society (2010), Providence, R.I..
 
 \bibitem{FeiPet}
 \newblock E. Feireisl, H. Petzeltov\'a,
\newblock \emph{On compactness of solutions to the Navier-Stokes equations of compressible flow},
\newblock J. Differential Equations, \textbf{163} (2000),  57--75.

\bibitem{FeiPet2}
\newblock E. Feireisl, A. Novotn\'y, H. Petzeltov\'a,
\newblock\emph{On the existence of globally defined weak solutions to the Navier-Stokes equations},
\newblock{J. Math. Fluid Mech., \textbf{3} (2001),  358--392.}


 
 \bibitem{GZ}
\newblock Z.H. Guo, C.J. Zhu, 
\newblock\emph{Global weak solutions and asymptotic behavior to 1D compressible Navier-Stokes equations with
density-dependent viscosity and vacuum}, 
\newblock J. Differential Equations \textbf{248} (2010) 2768--2799.
 
 \bibitem{HS}
 \newblock D. Hoff, D. Serre, 
 \newblock\emph{The failure of continuous dependence on initial data for the Navier-Stokes equations of compressible flow}, 
 \newblock SIAM J. Appl. Math. \textbf{51} (1991) 887--898.
 
 \bibitem{HSmo}
 \newblock D. Hoff, J. Smoller,
 \newblock\emph{Non-Formation of Vacuum States for Compressible Navier-Stokes Equations},
 \newblock Comm. Math. Physics \textbf{216} (2001) 255--276.
 
 
 
 \bibitem{JWX}
 \newblock Q. Jiu, Y. Wang, and Z. Xin,
 \newblock\emph{Global Well-Posedness of $2$D Compressible Navier-Stokes Equations
with Large Data and Vacuum},
\newblock J. Math. Fluid Mech. \textbf{16} (2014), 483--521.

\bibitem{JXZ}
\newblock S. Jiang, Z. Xin, P. Zhang, 
\newblock\emph{Global weak solutions to 1D compressible isentropic Navier-Stokes equations with density-dependent viscosity},
\newblock  Methods Appl. Anal. \textbf{12} (2005), 239--252. 

\bibitem{KV}
\newblock M. Kang, A. Vasseur,
\newblock \emph{ Global Smooth Solutions for 1D barotropic Navier-Stokes equations with a large class of degenerate viscosities},
\newblock Journal of Nonlinear Science \textbf{30} (2020), 1703--1721.

\bibitem{Kawa87}
 \newblock S. Kawashima,
 \newblock \emph{ Large-time behavior of solutions to hyperbolic-parabolic systems of conservation laws and applications},
 \newblock Proc. Roy. Soc. Edinburgh Sect. A \textbf{106} (1987), no. 1-2, 169--194.
 
 \bibitem{KNZ}
  \newblock S. Kawashima, S. Nishibata, P. Zhu,
 \newblock \emph{ Asymptotic Stability of the Stationary Solution
to the Compressible Navier-Stokes Equations
in the Half Space},
 \newblock Commun. Math. Phys. \textbf{240} (2003), 483--500.
 
 \bibitem{LZZ}
\newblock J.  Li, J. Zhang, J. Zhao, 
\newblock \emph{On the global motion of viscous compressible barotropic flows subject to large external potential forces and vacuum}, 
\newblock SIAM J. Math. Anal. \textbf{47} (2015), no. 2, 1121--1153. 

\bibitem{LiLiXin08}
 \newblock H.-L. Li, J. Li and Z. Xin,
 \newblock\emph { Vanishing of vacuum states and blow-up phenomena of the compressible Navier-Stokes equations},
 \newblock Comm. Math. Phys. \textbf{281} (2008), no. 2, 401--444.

\bibitem{LianGuoLi10}
 \newblock R. Lian, Z. Guo and H.-L. Li,
 \newblock\emph{ Dynamical behaviors for 1D compressible Navier-Stokes equations
with density-dependent viscosity},
 \newblock J. Differential Equations \textbf{248} (2010), no. 8, 1926--1954. 
 
 
 \bibitem{LinXu}
 \newblock Z. Lin, M. Xu,
 \newblock \emph{Metastability of Kolmogorov Flows and Inviscid Damping of Shear Flows},
 \newblock Arch.  Rational Mech.  Anal. \textbf{231} (2019), no. 3, 1811--1852.
 
 \bibitem{LXY}
\newblock T.P. Liu,  Z.P. Xin, T. Yang,
\newblock\emph{ Vacuum states of compressible flow}. 
\newblock Discrete Contin. Dyn. Syst. \textbf{4} (1998), 1--32.

\bibitem{Lio9698}
 \newblock P. L. Lions,
 \newblock\emph{ Topics in Fluids Mechanics}
Vol. 1 and 2, Oxford Lectures Series in Math. and its Appl.,
Oxford 1996 and 1998.




 
 
%
\bibitem{MascRou06}
 \newblock  C. Mascia and F. Rousset,
 \newblock \emph{ Asymptotic Stability of Steady-states for Saint-Venant Equations with Real Viscosity},
  \newblock  in "Analysis and simulation of fluid dynamics", (2007), 155--162,  Adv. Math. Fluid Mech., Birkhauser, Basel.
 

 
 \bibitem{MN}
 \newblock A. Matsumura and  K. Nishihara,
 \newblock \emph{Large-Time Behaviors of Solutions to an Inflow Problem in the Half Space for a One-Dimensional System of Compressible Viscous Gas},
 \newblock Commun. Math. Phys. \textbf{222} (2001), 449--474.
 
 \bibitem{MelVas}
 \newblock{A. Mellet, A. Vasseur,}
 \newblock \emph{ On the barotropic compressible Navier-Stokes equations}, 
 \newblock Comm. Partial Differential Equations  \textbf{32} (2007), no. 1-3, 431--452.
 
 
  \bibitem{MelVas2}
   \newblock{A. Mellet, A. Vasseur,}
 \newblock \emph{ Existence and Uniqueness of Global Strong Solutions for One-Dimensional Compressible Navier-Stokes Equations},
 \newblock SIAM J. Math. Anal.  \textbf{39} (2008), no. 4, 1344--1365.



 
 \bibitem{PS}
 \newblock P. Penel, I. Straskraba,
 \newblock \emph{Lyapunov analysis and stabilization to the rest state for solutions to the $1D$-barotropic compressible Navier-Stokes equations},
 \newblock C. R. Acad. Sci. Paris, Ser. I \textbf{345} (2007) 67--72.
 
 
 
 \bibitem{QHY}
 \newblock Y. Qin, L. Huang, Z. Yao,
 \newblock \emph{Regularity of $1$D compressible isentropic Navier-Stokes
equations with density-dependent viscosity},
  \newblock J. Differential Equations \textbf{245} (2008) 3956--3973.
  
 
 
\bibitem{Str14}
 \newblock M. Strani,
 \newblock \emph{ Existence and uniqueness of a positive connection for the scalar viscous shallow water system in a bounded interval},
 \newblock Comm. Pure Appl. Analysis, \textbf{13} (2014), no. 4, 1653--1667.  
 
 \bibitem{S16}
 \newblock M. Strani,
 \newblock\emph{Long time dynamics of layered solutions to the shallow water equations},
 \newblock Bull. Braz. Math. Soc. New Series, \textbf{47} (2016), no. 2, 1--13.
 

 
 \bibitem{STRAS}
 \newblock I. Straskraba, A. Zlotnik,
 \newblock \emph{On a decay rate for $1D$-viscous compressible barotropic fluid equations}
  \newblock J.evol.equ. \textbf{2} (2002) 69--96.
  
    \bibitem{STRAS2}
 \newblock I. Straskraba,
\newblock \emph{ Recent progress in the mathematical theory of $1D$ barotropic flow}
  \newblock Ann. Univ. Ferrara \textbf{55}  (2009), 395--405.




\bibitem{WangXu05}
 \newblock W. Wang and C.J. Xu,
 \newblock \emph{ The Cauchy problem for viscous Shallow Water flows},
 \newblock Rev. Mate. Iber. \textbf{21} (2005), 1--24.
 
 \bibitem{YYZ}
 \newblock T. Yang, Z.A. Yao, C.J. Zhu,
 \newblock\emph{ Compressible Navier-Stokes equations with density-dependent viscosity and vacuum}, 
 \newblock Comm. Partial Differential Equations \textbf{26} (5-6) (2001), 965--981.
 
 \bibitem{YZ}
 \newblock T. Yang, C.J. Zhu, 
 \newblock\emph{Compressible Navier--Stokes equations with degenerate viscosity coefficient and vacuum}, 
 \newblock Comm. Math. Phys. \textbf{230} (2002), no. 2,  329--363.
 
 
 \bibitem{ZZZ}
 \newblock T. Zheng, J. Zhang, J. Zhao,
 \newblock \emph{Asymptotic stability of viscous contact discontinuity to an inflow problem for compressible Navier-Stokes equations},
 \newblock Nonlinear Analysis \textbf{74} (2011), 6617--6639.
 




\end{thebibliography}
\end{document}